\documentclass[10pt]{article}

\usepackage{amssymb,amsmath,amsthm,amsfonts}
\usepackage{graphicx}
\usepackage[usenames]{color}
\usepackage{algorithm}
\usepackage{algpseudocode}

\newtheorem{thm}{Theorem}[section]
\newtheorem{rmk}{Remark}[section]

\newtheorem{lem}{Lemma}[section]
\newtheorem{ass}{Assumption}[section]
\newtheorem{prop}{Proposition}[section]
\newtheorem{cor}{Corollary}[section]
\newtheorem{exam}{Example}

\newcommand{\bs}{\boldsymbol}
\title{Multi-Parameter Tikhonov Regularization}
\date{\today}
\author{Kazufumi Ito\thanks{Center for Research in Scientific Computation \&
Department of Mathematics, North Carolina State University, Raleigh,
NC 27695, USA. (kito@math.ncsu.edu)} \quad\quad Bangti
Jin\thanks{Department of Mathematics and Institute for Applied Mathematics and Computational Science,
Texas A\&M University, College Station, 77843-3368 Texas, USA. ({\texttt
btjin@math.tamu.edu})}\quad\quad Tomoya Takeuchi\thanks{Center
for Research in Scientific Computation, North Carolina State
University, Raleigh, NC. ({\tt tntakeuc@ncsu.edu})}}
\begin{document}
\maketitle

\begin{abstract}
We study multi-parameter Tikhonov regularization, i.e., with multiple penalties. Such
models are useful when the sought-for solution exhibits several distinct features
simultaneously. Two choice rules, i.e., discrepancy principle and balancing principle,
are studied for choosing an appropriate (vector-valued) regularization parameter, and
some theoretical results are presented. In particular, the consistency of the discrepancy
principle as well as convergence rate are established, and an a posteriori error estimate
for the balancing principle is established. Also two fixed point algorithms are proposed
for computing the regularization parameter by the latter rule. Numerical results for
several nonsmooth multi-parameter models are presented, which show clearly their superior
performance over their single-parameter counterparts.
\end{abstract}

\section{Introduction}
In this paper, we are interested in solving linear inverse problems
\begin{equation}\label{eqn:lininv}
Kx=y^\delta,
\end{equation}
where $y^\delta\in Y$ is a noisy version of the exact data $y^\dagger=Kx^\dagger\in Y$
with $\delta^2=\phi(x^\dagger,y^\delta)$ being the noise level, the operator
$K:X\rightarrow Y$ is bounded and linear, and the spaces $X$ and $Y$ are Banach spaces.

Typically, problem \eqref{eqn:lininv} suffers from ill-posedness in the sense that a
small perturbation in the data might lead to large deviations in the retrieved solution,
and this often poses great challenges to their stable yet accurate numerical solution.
Usually, a regularization strategy is applied to find a stable approximate solution
\cite{TikhonovArsenin:1977,EnglHankeNeubauer:1996}. The most widely adopted approach is
Tikhonov regularization, which seeks an approximation $x_{\bs\eta}^\delta$ to problem
\eqref{eqn:lininv} by minimizing the following Tikhonov functional
\begin{equation*}
J_{\bs\eta}(x)=\phi(x,y^\delta)+\bs{\eta}\cdot\bs{\psi}(x).
\end{equation*}
Here the functionals $\phi$ and $\bs\psi$ represent data fidelity and (vector-valued)
penalty, respectively, and $\bs{\eta}\cdot\bs\psi(x)$ is the dot product between
$\bs{\eta}=(\eta_1,\ldots,\eta_n)^\mathrm{T}$ and $\bs\psi(x) =
(\psi_1(x),\ldots,\psi_n(x))^\mathrm{T}$, i.e.,
$\bs{\eta}\cdot\bs{\psi}(x)=\sum_{i=1}^n\eta_i\psi_i(x)$. Common choices of the fidelity
$\phi(x,y^\delta)$ include $\|Kx-y^\delta\|_{L^2}^2$, $\|Kx-y^\delta\|_{L^1}$ and
$\int(Kx-y^\delta\ln Kx)$, which are statistically well suited to additive Gaussian
noise, Laplace (impulsive) noise and Poisson noise, respectively. The penalties $\psi_i$
are nonnegative, convex and (weak) lower semicontinuous. The typical choice includes
$\|x\|_{L^2}^2$, $\|x\|_{\ell^p}^p$, $\|x\|_{H^m}^2$ and $|x|_{TV}$ etc. The
regularization parameter vector $\bs\eta$ compromises fidelity with penalties.

The use of multiple penalties, henceforth called multi-parameter regularization, in the
functional $J_{\bs\eta}$ is motivated by practical applications which exhibit
multiple/multiscale features. We just take microarray data analysis for an example. Here
the number of data is often far less than that of the unknowns. A desirable approach
should select all variables relevant to the proper functioning of gene network. The
conventional $\ell^2$ penalty tends to select all variables, including irrelevant ones,
since the resulting estimate has almost no nonzero entries. To remedy this issue,
$\ell^1$ penalty has been suggested as an alternative. However, the $\ell^1$ approach
delivers undesirable results for problems where there are highly correlated features and
all relevant ones are to be identified in that it tends to select only one feature out of
the relevant group instead of all relevant features of the group \cite{ZouHastie:2005},
thereby missing the groupwise structure. Zou and Hastie \cite{ZouHastie:2005} proposed
the elastic-net by incorporating the $\ell^2$ penalty into the $\ell^1$ penalty, in the
hope of retrieving the whole relevant group, and numerically demonstrated its excellent
performance for simulation studies and real-data applications. Such multiple/multiscale
features appear also in many other applications, e.g., image processing
\cite{LuShenXu:2007,Stephanakis:1997}, electrocardiography
\cite{BrooksAhmadMacLeodMaratos:1999}, and geodesy \cite{XuFukudaLiu:2006}.

A number of experimental studies
\cite{BrooksAhmadMacLeodMaratos:1999,ZouHastie:2005,XuFukudaLiu:2006} have shown great
potential of multi-parameter models for better capturing multiple distinct features of
the solution. However, a general theory for such models remains largely under-explored.
There are several attempts on various aspects, e.g., parameter choice, convergence and
statistical interpretation
\cite{BelgeKilmerMiller:2002,BrezinskiRedivoRodriguezSeatzu:2003,ChenLuXuYang:2008,JinLorenzSchiffler:2009,
DeMolVetoRosasco:2009,LuPereverzev:2009,LuPereverzevShaoTautenhahn:2010} of
multi-parameter regularization. For instance, Lu et al \cite{LuPereverzev:2009} discussed
the discrepancy principle using Hilbert space scales, and derived some error estimates,
but the parameter is vastly nonunique and it remains unclear which one to use. They also
adapted the model function approach to choose the regularization parameter, but the
underlying mechanism remains unclear. Jin et al \cite{JinLorenzSchiffler:2009} recently
investigated the properties, e.g., consistency and error estimates, of elastic-net for
asymptotically linear coupling between the two terms, and proposed two active-set type
methods for efficient numerical realization.

This paper aims at developing some theory for such models in a general framework. The
value function and its properties are first derived. Then two parameter choice rules,
i.e., discrepancy principle and balancing principle, are studied. The consistency and
convergence rates are established for the former. The balancing principle can be derived
from the Bayesian inference \cite{JinZou:2009}, and it was generalized in
\cite{ItoJinZou:2010}. The principle balances the penalty with the fidelity term. The
variant under consideration here is solely based on the value function, and does not
require a knowledge of the noise level. An a posteriori error estimate is derived, and
two efficient numerical algorithms are also proposed.

The rest of the paper is structured as follows. In Section \ref{sec:F}, we investigate
the value function and derive some properties, e.g., monotonicity, concavity, asymptotic
and especially differentiability. In Section \ref{sec:rule}, we investigate two parameter
choice rules, i.e., discrepancy principle and balancing principle, and discuss their
theoretical properties. In addition, two fixed point algorithms for the efficient
numerical realization of the balancing principle are proposed. Numerical results for
several examples are presented in Section \ref{sec:exp} to illustrate the efficiency and
accuracy of the proposed approaches. Finally, we conclude the paper with several future
research topics.

\paragraph{Notation} Let $x_{\bs\eta}^\delta$ be a minimizer to the functional $J_{\bs\eta}(x)$, and
$\mathcal{M}_{\bs\eta}$ be the set of minimizers. For vectors $\bs\eta\in\mathbb{R}^n$
and $\hat{\bs\eta}\in \mathbb{R}^n$, we denote by ${\bs\eta} \leq \hat{\bs\eta}$ if
$\eta_i \leq \hat{\eta}_i\,\forall1\leq i\leq n$.

\section{The value function and its properties}\label{sec:F}
In this section, we collect important properties of the value function $F(\bs\eta)$
defined by
\begin{equation}\label{valuef}
F(\bs\eta)=\inf_{x\in Q_{\text{ad}}}J_{\bs\eta}(x),
\end{equation}
where the set $Q_\mathrm{ad}$ stands for a convex constraint. Here, the existence of a
minimizer $x_{\bs\eta}^\delta$ to the functional $J_{\bs\eta}$ is not a priori assumed.
Provided that a minimizer $x_{\bs\eta}^\delta$ does exist, we have
$F(\bs\eta)=J_{\bs\eta} (x_{\bs\eta}^\delta)$. The value function $F$ will play an
important role in developing a balancing principle, see Section \ref{subsec:bal}. The
results presented below generalize those for the single parameter \cite{Ito:2010a}, and
the proofs are similar and thus omitted.

A first result shows the continuity and concavity of $F$.
\begin{lem}\label{lem:F}
The value function $F(\bs\eta)$ is monotonically increasing in the sense
$F(\hat{\bs\eta})\le F(\bs\eta) $ if $\hat{\bs\eta} \le \bs\eta$, and it is concave.
\end{lem}

\begin{rmk}
Lemma \ref{lem:F} does not require the existence of $x\in Q_{\mathrm{ad}}$ that achieves
the infimum of $J_{\bs\eta}$. The results are also true for nonlinear operators and in
the presence of convex constraint $Q_\mathrm{ad}$.
\end{rmk}

Next we examine the properties of the value function $F$ more closely. Recall first
one-sided partial derivatives $\partial^\pm_i F$ are defined by
\begin{equation*}
\begin{aligned}
\partial^{-}_iF(\bs{\eta}) =
\lim_{h\rightarrow 0^+} \frac{F({\bs\eta})-F({\bs\eta}-h\bs{e}_i)}{h},\\
\partial^{+}_iF({\bs\eta}) = \lim_{h \rightarrow 0^+} \frac{F({\bs\eta}+h\bs{e}_i)-F({\bs\eta})}{h},
\end{aligned}
\end{equation*}
where $\bs e_i$ is the $i$th canonical basis.

The next result shows some properties, i.e., existence, nonnegativity, monotonicity and
(left- and right-) continuity, of the one-side partial derivatives $\partial_i^\pm F$.
The properties follow directly from Lemma \ref{lem:F}.
\begin{lem}For any $\bs\eta>0$, there hold
\begin{itemize}
  \item[$(i)$] The one-sided partial derivatives $\partial_i^\pm F(\bs\eta)$ exist,
      and $\partial_i^{\pm}F(\bs\eta)\ge 0$;
  \item[$(ii)$] For any $h>0$, there holds $0\leq\partial_i^{+}F(\bs\eta+h\bs
      e_i)\leq \partial_i^{-}F(\bs\eta+h\bs e_i) \leq
      \partial_i^{+}F(\bs\eta)\leq \partial_i^{-}F(\bs\eta)$;
  \item[$(iii)$] $\displaystyle\partial_i^{-}F(\bs\eta)= \lim_{h\rightarrow0^+}
      \partial_i^{-}F(\bs\eta-h\bs e_i)$ and $\displaystyle \partial_i^{+}F(\bs\eta)
      = \lim_{h \rightarrow 0^+}\partial_i^{+}F(\bs\eta+h\bs e_i)$.
\end{itemize}
\end{lem}

\begin{rmk} The partial differentiability of $F$ in the $i$-th direction at ${\bs{\eta}}$
guarantees the continuity of $\partial_i^{\pm}F$ at this point. Indeed, the monotonicity
of $\partial_i^{\pm}F$ and the left continuity of $\partial_i^+F$ yield the inequalities
\begin{equation*}
\partial_i^+F({\bs{\eta}}) = \lim_{h\rightarrow 0^+}\partial_i^+F({\bs{\eta}}+h\bs
e_i)\le \lim_{h\rightarrow 0^+}\partial_i^-F({\bs{\eta}}+h\bs e_i) \le
\partial_i^-F({\bs{\eta}}).
\end{equation*}
Now suppose $F$ is differentiable at ${\bs{\eta}}$, i.e.,
$\partial_i^{-}F({\bs{\eta}})=\partial_i^{+}F({\bs{\eta}})$. Then from the inequalities
it follows that
\begin{equation*}
\lim_{h\rightarrow 0^+}\partial_i^-F({\bs{\eta}}+h\bs e_i) =\partial_i^-F({\bs{\eta}}),
\end{equation*}
which shows the continuity of $\partial_i^{-}F$ at ${\bs{\eta}}$. Similarly it follows
that $\partial_i^+F$ is continuous at ${\bs{\eta}}$.
\end{rmk}

The asymptotic behavior of $F(\bs\eta)$ is useful for designing numerical algorithms.
\begin{prop}\label{prop:asymp}
The following asymptotics of $F$ hold
\begin{equation*}
\lim_{|\bs\eta|\rightarrow0}F(\bs\eta)=\inf_{x\in
Q_{\text{ad}}}\phi(x,y^\delta) \quad\mbox{and}\quad \lim_{|\bs\eta|\rightarrow
0}\eta_i \partial_i^{\pm}F(\bs\eta)=0.
\end{equation*}
\end{prop}

The partial derivatives $\partial_i^\pm F$ are closely connected to the fidelity $\phi$
and penalty $\bs\psi$ under the assumption of existence of a minimizer, i.e., the set
$\mathcal{M}_{\bs\eta}$ is nonempty. This is guaranteed by:
\begin{ass}\label{ass:phipsi} The functionals $\phi$ and $\psi_i$ satisfy:
\begin{itemize}
\item[$(i)$] For any sequence $\{x_n\}_n\subset Q_\mathrm{ad}$ such that $\phi$ and
    $\psi_i$ for all $1\le i \le n$ are uniformly bounded, there exists a subsequence
    $\{x_{n_k}\}_k$ which converges to an element $x^*\in Q_{ad}$ in the
    $\tau$-topology.
\item[$(ii)$] $\phi$ and $\psi_i$ are lower semi-continuous with respect to
    $\tau$-convergent sequences, i.e., if a subsequence $\{x_{n}\}_n$ converges to
    $x^*\in Q_\mathrm{ad}$ in $\tau$-topology, then
\begin{equation*}
\phi(x^*)\leq\liminf_{n \to \infty}\phi(x_{n}) \mbox{ and }\,
\psi_i(x^*)\leq\liminf_{n \to \infty }\psi_i(x_{n}).
\end{equation*}
\end{itemize}
\end{ass}

In case that the set $\mathcal{M}_{\bs\eta}$ contains multiple elements, there might
exist distinct $x_{\bs\eta}^\delta,\hat{x}_{\bs\eta}^\delta\in\mathcal{M}_\eta$ such that
\begin{equation*}
F(\eta)=\phi(x_{\bs\eta}^\delta,y^\delta)+\bs\eta\cdot\bs\psi(x_{\bs\eta}^\delta)=\phi(\hat
x_{\bs\eta}^\delta,y^\delta)+\bs\eta\cdot\bs\psi(\hat x_{\bs\eta}^\delta)
\quad \mbox{but}\quad
\phi(x_{\bs\eta}^\delta,y^\delta)\neq\phi(\hat x_{\bs\eta}^\delta,y^\delta),
\end{equation*}
i.e., the functions $\phi(x_{\bs\eta}^\delta,y^\delta)$ and $\bs\psi(x_{\bs\eta}^\delta)$
are potentially multi-valued in $\bs\eta$.

A first relation between $\psi_i$ and $\partial_i^{\pm} F$ is given by
\begin{lem}\label{lem:DF}
Let Assumption \ref{ass:phipsi} be fulfilled. Then for any
$x_{\bs{\eta}}^\delta\in\mathcal{M}_{\bs{\eta}}$, there hold
\begin{align*}
\partial_i^{+}F({\bs{\eta}})\leq\ &\psi_i(x_{\bs{\eta}}^\delta) \leq\partial_i^{-}F({\bs{\eta}})\quad i=1,\ldots,n,\\
F({\bs{\eta}})-\sum_{i=1}^n\eta_i\partial_i^{-}F(\bs\eta)\leq \
&\phi(x_{\bs\eta}^\delta,y^\delta) \leq F(\bs\eta)-\sum_{i=1}^n\eta_i
\partial_i^{+}F(\bs\eta).
\end{align*}
\end{lem}

An immediate consequence of Lemma \ref{lem:DF} is:
\begin{cor}\label{cor:DF=psi}
Let Assumption \ref{ass:phipsi} be fulfilled. If $\partial_iF(\bs\eta)$ exists $\bs\eta$
for all $i$, then $\psi_i(x_{\bs\eta}^\delta)$ and $\phi(x_{\bs\eta}^\delta,y^\delta)$
are single valued at $\bs\eta$ and
\begin{equation*}
  \partial_iF(\bs\eta)=\psi_i(x_{\bs\eta}^\delta) \quad \mbox{and}\quad
  F(\bs\eta)-\bs\eta\cdot\partial F(\bs\eta)=\phi(x_{\bs\eta}^\delta,y^\delta)\quad\forall
  x_{\bs\eta}^\delta\in\mathcal{M}_{\bs\eta}.
\end{equation*}
\end{cor}

More precisely, the partial derivatives $\partial^\pm_i$ can be expressed by $\psi_i$ as
follows.
\begin{thm}\label{thm:difful}
Let Assumption \ref{ass:phipsi} hold. Then for any $\bs\eta>0$ and every $i$, there exist
$x^+_i,x^-_i\in\mathcal{M}_{\bs\eta}$ such that
\begin{equation*}
\psi_i(x^+_i)=\partial_i^+F(\bs\eta)\quad\mbox{and}\quad \psi_i(x^-_i)=\partial_i^-F(\bs\eta).
\end{equation*}
\end{thm}

Theorem \ref{thm:difful} in conjunction with Lemma \ref{lem:DF} implies the following
corollary.
\begin{cor}Let Assumption \ref{ass:phipsi} hold. Then
\begin{itemize}
  \item[$(i)$] There exist $x^{+}_i,x^{-}_i\in\mathcal{M}_{\bs\eta}$ such that
      $\displaystyle \psi_i(x^{+}_i) = \min_{x \in \mathcal{M}_{\bs{\eta}}}\psi_i(x)$
      and $\displaystyle \psi_i(x^{-}_i) = \max_{x \in
     \mathcal{M}_{\bs{\eta}}}\psi_i(x)$.
  \item[$(ii)$] If $\psi_i(x_{\bs{\eta}}^\delta)=\psi_i(\hat{x}_{\bs{\eta}}^\delta)$
      for all $x_{\bs{\eta}}^\delta, \hat{x}_{\bs{\eta}}^\delta\in
      \mathcal{M}_{\bs{\eta}}$ for all ${\bs{\eta}}>0$, then
      $\partial_iF({\bs{\eta}})$ exists and it is continuous.
\end{itemize}
\end{cor}

The last result gives a sufficient condition for the differentiability of the value
function $F$. It plays an important role in especially designing an efficient algorithm
for certain choice rules, by e.g., Morozov's principle and balancing principle
\cite{JinZou:2010}.
\begin{thm}\label{thm:uniquediff}
Assume that the minimizer of the functional $J_{\bs\eta}$ is unique at ${\bs{\eta}}>0$.
Then the derivatives $\{\partial_iF({\bs{\eta}})\}_i$ exist and are continuous at
$\bs\eta$. In particular, $F$ is differentiable at ${\bs{\eta}}$.
\end{thm}

\section{Parameter choice rules}\label{sec:rule}
In this section, we discuss two choice rules, i.e., discrepancy principle
\cite{Morozov:1966,JinZou:2009} and balancing principle \cite{Ito:2010a}, for
multi-parameter models. For notational simplicity, we shall restrict our attention to the
case of two penalty terms.

\subsection{Discrepancy principle}

Here we investigate the discrepancy principle due to Morozov \cite{Morozov:1966} for
multi-parameter regularization. We shall assume a triangle-type inequality for the
functional $\phi$.
\begin{ass}\label{ass:phi}
The functional $\phi(x,y)$ vanishes if and only if $Kx=y$, and satisfies an inequality
$\phi(x,y)\leq c(\phi(x',y')+\phi(x,y'))$ for some constant $c$ and any $x'$ with
$Kx'=y$.
\end{ass}

The discrepancy principle determines an appropriate (vector-valued) regularization
parameter $\bs\eta$ by
\begin{equation}\label{eqn:morozov}
\phi(x_{\bs\eta}^\delta,y^\delta)=c_m\delta^2
\end{equation}
for some constant $c_m\geq1$. The rationale of the principle is that the solution
accuracy in terms of the residual should be compatible with the data accuracy (noise
level).

\begin{thm}\label{thm:morozov1}
Let Assumptions \ref{ass:phipsi} and \ref{ass:phi} be satisfied and the operator $K$ be
injective. Then for any $\bs\eta\equiv\bs\eta(\delta)$ satisfying \eqref{eqn:morozov} and
$c_0\leq\tfrac{\eta_1(\delta)}{\eta_2(\delta)}\leq c_1$ for some $c_0,c_1>0$, there holds
$\lim_{\delta\rightarrow0}x_{\bs\eta}^\delta=x^\dagger$ in $\tau$-topology.
\end{thm}
\begin{proof}
The minimizing property of $x_{\bs\eta}^\delta$ implies
\begin{equation*}
\begin{aligned}
\phi(x_{\bs\eta}^\delta,y^\delta)+\bs\eta\cdot\bs\psi(x_{\bs\eta}^\delta)
&\leq\phi(x^\dagger,y^\delta)+\bs\eta\cdot\bs\psi(x^\dagger)\\
&\leq\delta^2+\bs\eta\cdot\bs\psi(x^\dagger).
\end{aligned}
\end{equation*}
From the discrepancy equation \eqref{eqn:morozov}, we deduce
\begin{equation}\label{ieqn:morozov}
\bs\eta\cdot\bs\psi(x_{\bs\eta}^\delta)\leq\bs\eta\cdot\bs\psi(x^\dagger).
\end{equation}
Therefore, either $\psi_1(x_{\bs\eta}^\delta)\leq \psi_1(x^\dagger)$ or
$\psi_2(x_{\bs\eta}^\delta)\leq \psi_2(x^\dagger)$ holds. Now the assumption
$c_0\leq\frac{\eta_1(\delta)}{\eta_2(\delta)}\leq c_1$ implies that the sequence
$\{\psi_i(x_{\bs\eta}^\delta),i=1,2\}_\delta$ is uniformly bounded. Hence the coercivity
of the functional indicates that the sequence $\{x_{\bs\eta}^\delta\}_\delta$ is
uniformly bounded. Thus there exists a subsequence, also denoted by
$\{x_{\bs\eta}^\delta\}_\delta$, and some $x^\ast$, such that
\begin{equation*}
x_{\bs\eta}^\delta\rightarrow x^\ast\quad \mbox{ in $\tau$-topology}.
\end{equation*}
The $\tau$-lower semicontinuity of the functional $\phi$ and Assumption \ref{ass:phi}
yields
\begin{equation*}
0\leq\phi(x^\ast,y^\dagger)\leq
c\liminf_{\delta\rightarrow0}(\phi(x^\dagger,y^\delta)+\phi(x_{\bs\eta}^\delta,y^\delta))\leq
\liminf_{\delta\rightarrow0}c(1+c_m)\delta^2=0.
\end{equation*}
In particular, $\phi(x^\ast,y^\dagger)=0$, i.e., $Kx^\ast=y^\dagger$. This together with
the injectivity $K$ implies $x^\ast=x^\dagger$. Since every subsequence has a
subsubsequence converging to $x^\dagger$, the whole sequence converges to $x^\dagger$.
\end{proof}

\begin{rmk}
The condition $c_0\leq\tfrac{\eta_1(\delta)}{\eta_2(\delta)}\leq c_1$ ensures the uniform
boundedness of both penalties, and thus we can utilize the lower-semicontinuity of the
functionals to arrive at the desired $\tau$-convergence.
\end{rmk}

\begin{thm}\label{thm:morozov2} Let Assumptions \ref{ass:phipsi} and \ref{ass:phi} hold. If
a subsequence $\{\bs\eta(\delta)\}_\delta$ converges and
$\displaystyle\tilde{\eta}\equiv\lim_{\delta\rightarrow0}\tfrac{\eta_1(\delta)}{\eta_2(\delta)}>0$.
Then it contains a subsequence $\tau$-converging to an
$\tilde\eta\psi_1+\psi_2$-minimizing solution and
\begin{equation*}
\lim_{\delta\rightarrow0}\left(\frac{\eta_1(\delta)}{\eta_2(\delta)}\psi_1(x_{\bs\eta}^\delta)+\psi_2(x_{\bs\eta}^\delta)
\right)=\tilde\eta\psi_1(x^\dagger)+\psi_2(x^\dagger).
\end{equation*}
Moreover, if the $\tilde\eta\psi_1+\psi_2$-minimizing solution is unique, then the whole
subsequence $\tau$-converges.
\end{thm}
\begin{proof}
By repeating the arguments in Theorem \ref{thm:morozov1}, we deduce that there exists a
subsequence, also denoted by $\{x_{\bs\eta}^\delta\}_\delta$, and some $x^\ast$, such
that
\begin{equation*}
x_{\bs\eta}^\delta\rightarrow x^\ast\quad \mbox{ in $\tau$-topology}.
\end{equation*}
and by the $\tau$-lower-semicontinuity, we have $\phi(x^\ast,y^\dagger)=0$. By virtue of
lower semicontinuity of the functionals and inequality \eqref{ieqn:morozov}, we deduce
\begin{equation*}
\begin{aligned}
\tilde\eta\psi_1(x^\ast)+\psi_2(x^\ast)&\leq\liminf_{\delta\rightarrow0}\left(\frac{\eta_1(\delta)}{\eta_2(\delta)}
\psi_1(x_{\bs\eta}^\delta)+\psi_2(x_{\bs\eta}^\delta)\right)\\
&\leq\limsup_{\delta\rightarrow0}\left(\frac{\eta_1(\delta)}{\eta_2(\delta)}\psi_1(x_{\bs\eta}^\delta)+\psi_2(x_{\bs\eta}^\delta)\right)\\
&\leq\lim_{\delta\rightarrow0}\left(\frac{\eta_1(\delta)}{\eta_2(\delta)}\psi_1(x^\dagger)+\psi_2(x^\dagger)\right)
=\tilde\eta\psi_1(x^\dagger)+\psi_2(x^\dagger),
\end{aligned}
\end{equation*}
which together with the identity $\phi(x^\ast,y^\dagger)=0$ implies that $x^\ast$ is an
$\tilde\eta\psi_1+\psi_2$-minimizing solution. The desired identity follows from the
above inequalities with $x^\ast$ in place of $x^\dagger$. The whole sequence convergence
follows from the standard subsequence argument.
\end{proof}

In Theorems \ref{thm:morozov1} and \ref{thm:morozov2}, we have assumed the existence of a
solution $\bs\eta$ to equation \eqref{eqn:morozov}. This is guaranteed if the Tikhonov
functional $J_{\bs\eta}$ has a unique minimizer, see Theorem \ref{thm:uniquediff}.
\begin{thm}
Assume that $J_{\bs\eta}$ has a unique minimizer for all $\bs\eta>0$,
$\lim_{|\bs\eta|\rightarrow0}\phi(x_{\bs\eta}^\delta,y^\delta)<c_m\delta^2$, and there is
a sequence $\{\bs\eta_n\}$ such that $\lim_{n\rightarrow\infty}
\phi(x_{\bs\eta_n}^\delta, y^\delta)>c_m\delta^2$. Then there exists at least one
solution to \eqref{eqn:morozov}.
\end{thm}
\begin{proof}
By Theorem \ref{thm:uniquediff} and Lemma \ref{lem:F}, the uniqueness of a minimizer to
$J_{\bs\eta}$ for all $\bs\eta>0$ implies that the function
$\phi(x_{\bs\eta}^\delta,y^\delta)$ is continuous in $\bs\eta$. The desired assertion
follows from the continuity directly.
\end{proof}

Lastly, we present an error estimate in case of $Y$ being a Hilbert space and
$\phi(x,y^\delta)=\|Kx-y^\delta\|^2$ and convex penalties $\bs\psi$. We use Bregman
distance to measure the error. Denote the subdifferential of a functional $\psi(x)$ at
$x^\dagger$ by $\partial\psi(x^\dagger)$, i.e., $\partial\psi(x^\dagger)=\{\xi\in X^\ast:
\psi(x)\geq\psi(x^\dagger)+\langle \xi,x-x^\dagger\rangle\, \forall x\in X\}$, and the
Bregman distance $d_\xi(x,x^\dagger)$ by for any $\xi\in\partial\psi(x^\dagger)$
\begin{equation*}
d_\xi(x,x^\dagger):=\psi(x)-\psi(x^\dagger)-\langle
\xi,x-x^\dagger\rangle.
\end{equation*}

\begin{thm}\label{thm:morozovest1}
If $Y$ is a Hilbert space and the exact solution $x^\dagger$ satisfies the source
condition: $\mathrm{range}(K^\ast)\cap\partial\psi_1
(x^\dagger)\cap\partial\psi_2(x^\dagger)\neq\emptyset$. Then for any $\bs\eta^\ast$
solving \eqref{eqn:morozov}, there exists some $i$ and $\xi_i\in
\partial\psi_i(x^\dagger)$ such that
\begin{equation*}
d_{\xi_i}(x_{\bs\eta^\ast}^\delta,x^\dagger)\leq C\delta.
\end{equation*}
\end{thm}
\begin{proof}
By the minimizing property of $x_{\bs\eta^\ast}^\delta$, we have
\begin{equation*}
\phi(x_{\bs\eta^\ast}^\delta,y^\delta)+\bs\eta^\ast\cdot\bs\psi(x_{\bs\eta^\ast}^\delta)
\leq\phi(x^\dagger,y^\delta)+\bs\eta^\ast\cdot\bs\psi(x^\dagger)\leq\delta^2+\bs\eta^\ast\cdot\bs\psi(x^\dagger).
\end{equation*}
The definition of the discrepancy principle indicates
\begin{equation*}
\bs\eta^\ast\cdot\bs\psi(x_{\bs\eta^\ast}^\delta)\leq
\bs\eta^\ast\cdot\bs\psi(x^\dagger).
\end{equation*}
Consequently, we have that there holds $\psi_i(x_{\bs\eta^\ast}^\delta)
\leq\psi_i(x^\dagger)$ for either $i=1$ or $i=2$. Therefore, by the source condition, for
some $\xi_i\in\mathrm{range}(K^\ast)\cap\partial\psi_i(x^\dagger)$ or equivalently
$\xi_i=K^\ast w_i$ for some source representer $w_i$, and the Cauchy-Schwarz inequality,
we deduce
\begin{equation*}
\begin{aligned}
d_{\xi_i}(x_{\bs\eta^\ast}^\delta,x^\dagger)&=\psi_i(x_{\bs\eta^\ast}^\delta)-\psi_i(x^\dagger)-\langle
\xi_i,x_{\bs\eta^\ast}^\delta-x^\dagger\rangle\leq-\langle\xi_i,x_{\bs\eta^\ast}^\delta-x^\dagger\rangle\\
&=-\langle K^\ast
w_i,x_{\bs\eta^\ast}^\delta-x^\dagger\rangle=-\langle
w_i,K(x_{\bs\eta^\ast}^\delta-x^\dagger)\rangle\\
&\leq\|w_i\|\|K(x_{\bs\eta^\ast}^\delta-x^\dagger)\|\\
&\leq\|w_i\|\left(\|Kx_{\bs\eta^\ast}^\delta-y^\delta\|+\|y^\delta-Kx^\dagger\|\right)\leq
(1+c_m)\|w_i\|\delta.
\end{aligned}
\end{equation*}
This shows the desired estimate.
\end{proof}

The source condition in Theorem \ref{thm:morozovest1} can be hard to argue.
Alternatively, we can have another convergence rates result under a seemingly less
restrictive assumption.
\begin{thm}
If $Y$ is a Hilbert space and the exact solution $x^\dagger$ satisfies the source
condition: for any $t\in[0,1]$, there exists $w_t$ such that $K^\ast w_t=\xi_t\in
\partial(t\psi_1(x^\dagger)+(1-t)\psi_2(x^\dagger))$. Then for any
$\bs\eta^\ast$ solving \eqref{eqn:morozov}, and letting
$t^*=\tfrac{\eta_1^\ast(\delta)}{\eta_1^\ast(\delta)+\eta_2^\ast(\delta)}$, the following
estimate holds
\begin{equation*}
d_{\xi_{t^*}}(x_{\bs\eta^\ast}^\delta,x^\dagger)\leq C\delta.
\end{equation*}
\end{thm}
\begin{proof}
By the minimizing property of $x_{\bs\eta^\ast}^\delta$, we have
\begin{equation*}
t^*\psi_1(x_{\bs\eta^\ast}^\delta) + (1-t^*)\psi_2(x_{\bs\eta^\ast}^\delta)  \leq
t^*\psi_1(x^\dagger)+(1-t^*)\psi_2(x^\dagger).
\end{equation*}
Therefore, by the source condition, for some $\xi_{t^*}\in\partial(t^*\psi(x^\dagger)
+(1-t^*)\psi(x^\dagger))$ and $w_{t^*}\in Y$ such that $\xi_{t^*}=K^\ast w_{t^*}$, and
the Cauchy-Schwarz inequality, we deduce
\begin{equation*}
\begin{aligned}
d_{\xi_{t^*}}(x_{\bs\eta^\ast}^\delta,x^\dagger)&=(t^\ast\psi_1(x_{\bs\eta^\ast}^\delta)+(1-t^\ast)\psi_2(x_{\bs\eta^\ast}^\delta))
-(t^\ast\psi_1(x^\dagger)+(1-t^\ast)\psi_2(x^\dagger))-\langle
\xi_{t^*},x_{\bs\eta^\ast}^\delta-x^\dagger\rangle\\
&\leq-\langle\xi_{t^*},x_{\bs\eta^\ast}^\delta-x^\dagger\rangle=-\langle K^\ast w_{t^*},x_{\bs\eta^\ast}^\delta-x^\dagger\rangle\\
&=-\langle w_{t^*},K(x_{\bs\eta^\ast}^\delta-x^\dagger)\rangle
\leq\|w_{t^*}\|\|K(x_{\bs\eta^\ast}^\delta-x^\dagger)\|\\
&\leq\|w_{t^*}\|\left(\|Kx_{\bs\eta^\ast}^\delta-y^\delta\|+\|y^\delta-Kx^\dagger\|\right)\leq
(1+c_m)\|w_{t^*}\|\delta.
\end{aligned}
\end{equation*}
This shows the desired estimate.
\end{proof}
\begin{rmk}
In the practical applications of the discrepancy principle, one needs to find the
solution of a nonlinear equation in $\bs\eta$. The uniqueness of a solution to equation
\eqref{eqn:morozov} is not guaranteed, and additional conditions need to be supplied for
definiteness. Lastly, we would like to mention that the principle can be efficiently
realized by the model function approach \cite{JinZou:2010}.
\end{rmk}

\subsection{Balancing principle}\label{subsec:bal}

The discrepancy principle described earlier requires an estimate of the noise level
$\delta$, which is not always available in practical applications. Therefore, it is of
great interest to develop heuristic rules that do not require this knowledge. One such
rule is the balancing principle, for which there are several variants, see
\cite{ItoJinZou:2010} for details. The principle can be derived from the augmented
Tikhonov (a-Tikhonov) regularization \cite{JinZou:2009}, which admits clear statistical
interpretations as hierarchical Bayesian modeling. In particular, it provides the
mechanism to automatically balance the penalty with the fidelity, see also Remark
\ref{rmk:bal}. The variant under consideration is due to \cite{Ito:2010a}, and has
demonstrate very promising empirical results for several common single-parameter models
\cite{Ito:2010a}. Finally we remind the balancing principle discussed here should not be
confused with the principle due to Lepskii which is sometimes also named balancing
principle \cite{Mathe:2006} and does require a precise knowledge of the noise level.

First we first sketch the a-Tikhonov regularization approach. For multi-parameter models,
it can be derived analogously from Bayesian inference \cite{JinZou:2009,ItoJinZou:2010},
and the resulting a-Tikhonov functional $J(x,\tau,\{\lambda_i\})$ is given by
\begin{equation*}
    J(x,\tau,\{\lambda_i\})=\tau\phi(x,y^\delta)+ \bs{\lambda}\cdot\bs{\psi}(x)+
    \sum_i(\beta_i\lambda_i-\alpha_i\ln\lambda_i) + \beta_0\tau-\alpha_0 \ln\tau,
\end{equation*}
which maximizes the posteriori probability density function
$p(x,\tau,\{\lambda_i\}|y^\delta)\propto\, p(y^\delta|x,\tau,\{\lambda_i\})\
p(x,\tau,\{\lambda_i\})$ under the assumption that the scalars $\{\lambda_i\}$ and $\tau$
have the Gamma distribution with known parameter pairs $(\alpha_i,\beta_i)$ and
$(\alpha_0,\beta_0)$, respectively. Let $\eta_i=\frac{\lambda_i}{\tau}$. Then the
necessary optimality condition of the a-Tikhonov functional is given by
\begin{equation*}
  \left\{\begin{aligned}
     x^\delta_{\bs\eta}&=\arg\min_{x}\;\left\{\phi(x,y^\delta)+\bs{\eta}\cdot\bs{\psi}(x)\right\},\\
        \lambda_i&=\dfrac{\alpha_i}{\psi_i(x^\delta_{\bs\eta})+\beta_i},\\
        \tau&=\dfrac{\alpha_0}{\phi(x_{\bs\eta}^\delta,y^\delta)+\beta_0}.
  \end{aligned}\right.
\end{equation*}
Upon assuming $\alpha_i=\alpha$ and $\beta_i=\beta$ for simplicity and letting
$\gamma=\frac{\alpha_0}{\alpha}$, then we have the following system for
$(x^\delta_{\bs\eta},{\bs\eta})$
\begin{equation}\label{eqn:atikhopt}
  \left\{\begin{aligned}
    x^\delta_{\bs\eta}&=\arg\min_{x}\;\left\{\phi(x,y^\delta)+{\bs\eta}\cdot{\bs\psi}(x)\right\},\\
    \eta_i&=\frac{1}{\gamma}\,\dfrac{\phi(x^\delta_{\bs\eta},y^\delta)+\beta_0}{\psi_i(x^\delta_{\bs\eta})+\beta}.
  \end{aligned}\right.
\end{equation}

Next, we give the promised balancing principle. The multi-parameter counterpart of the
balancing principle given in \cite{Ito:2010a} consists of minimizing
\begin{equation*}
 \Phi_\gamma({\bs{\eta}})=c_\gamma\frac{F^{2+\gamma}(\bs\eta)}{\eta_1\eta_2},
\end{equation*}
where the constant $c_\gamma=\frac{\gamma^\gamma}{(\gamma+2)^{\gamma+2}}$. We note that
this constant $c_\gamma$ can be quite arbitrary, except for comparison with the criterion
$\Psi_\gamma$ defined next. Another variant of the balancing principle reads
\begin{equation*}
\Psi_\gamma({\bs{\eta}})=\phi(x_{\bs\eta}^\delta,y^\delta)^\gamma\psi_1(x_{\bs\eta}^\delta)\psi_2(x_{\bs\eta}^\delta),
\end{equation*}
which generalizes a criterion due to Reginska \cite{EnglHankeNeubauer:1996}.

The relation between $\Phi_\gamma$ and $\Psi_\gamma$ is made explicit in the following
result.
\begin{prop}
Let the value function $F$ be twice continuously differentiable, $\partial_i F (i=1,2)$
do not vanish, and the Hessian $\nabla^2F$ be nonsingular. Then the criteria
$\Phi_\gamma$ and $\Psi_\gamma$ share the set of critical points, which are the solutions
to the system
\begin{equation}\label{eqn:bal}
 \gamma\eta_1\psi_1(x_{\bs\eta}^\delta)=\gamma\eta_2\psi_2(x_{\bs\eta}^\delta)=\phi(x_{\bs\eta}^\delta,y^\delta).
\end{equation}
\end{prop}
\begin{proof}
Setting the first-order derivatives of the criterion $\Phi_\gamma$ to zero gives
\begin{equation*}
\nabla\Phi_\gamma(\eta)=c_\gamma\frac{F^{1+\gamma}(\bs\eta)}{\eta_1\eta_2}\left[\begin{array}{c}
(2+\gamma)\psi_1(x_{\bs\eta}^\delta)-\tfrac{F}{\eta_1}\\[1ex]
(2+\gamma)\psi_2(x_{\bs\eta}^\delta)-\tfrac{F}{\eta_2}
\end{array}\right]=\bs{0}.
\end{equation*}
This together with Lemma \ref{lem:DF} gives
$(2+\gamma)\eta_i\psi_i(x_{\bs\eta}^\delta)=F,\,i=1,2$. Consequently,
$\eta_1\psi_1(x_{\bs\eta}^\delta)=\eta_2\psi_2(x_{\bs\eta}^\delta)$, and thus system
\eqref{eqn:bal} holds. Meanwhile, by setting the first-order derivatives
$\nabla\Psi_\gamma(\bs\eta)$ of the the criterion $\Psi_\gamma$ to zero and noting Lemma
\ref{lem:DF}, we get
\begin{equation*}
\phi(x_{\bs\eta}^\delta,y^\delta)^{\gamma-1}\nabla^2F\left[\begin{array}{cc}
\psi_2(x_{\bs\eta}^\delta)&\\ &\psi_1(x_{\bs\eta}^\delta)
\end{array}\right]\left[\begin{array}{c}
-\gamma\eta_1\psi_1(x_{\bs\eta}^\delta)+\phi(x_{\bs\eta}^\delta,y^\delta)\\[1ex]
-\gamma\eta_2\psi_2(x_{\bs\eta}^\delta)+\phi(x_{\bs\eta}^\delta,y^\delta)
\end{array} \right]=\bs{0}.
\end{equation*}
By the assumption that the Hessian $\nabla^2F$ is nonsingular and
$\psi_i(x_{\bs\eta}^\delta)(i=1,2)$ do not vanish, we arrive at
\begin{equation*}
\begin{aligned}
\gamma\eta_i\psi_i(x_{\bs\eta}^\delta)-\phi(x_{\bs\eta}^\delta,y^\delta)=0,
\end{aligned}
\end{equation*}
i.e., system \eqref{eqn:bal}. This concludes the proof.
\end{proof}

\begin{rmk}
Criterion $\Phi_\gamma$ makes only use of the value function $F(\eta)$, not of the
derivatives of $F(\bs\eta)$, which can be potentially multi-valued in case that the
functional $J_{\bs\eta}$ has multiple minimizers. In contrast, the value function
$F(\bs\eta)$ is always continuous, see Lemma \ref{lem:F}, and thus the optimization
problem of minimizing $\Phi_\gamma$ over any bounded regions is always well-defined.  For
models with potentially nonunique minimizers, criterion $\Psi_\gamma$ and balancing
principle, i.e., equation \eqref{eqn:bal}, are ill-defined, and the corresponding
minimization formulations can be problematic. The criterion $\Phi_\gamma$ is advantageous
then.
\end{rmk}

\begin{rmk}\label{rmk:bal}
Balancing principle is named after system \eqref{eqn:bal}: it attempts to balance the
fidelity with the penalties with the parameter $\gamma$ being the relative weight.
Comparing \eqref{eqn:bal} with \eqref{eqn:atikhopt} shows clearly the intimate
connections between the a-Tikhonov approach and the balancing principle: the a-Tikhonov
approach builds in the principle automatically, and consequently the hierarchical
Bayesian modeling is also balancing. Finally, we would like to remark that the balancing
idea has been developed from other perspectives, see \cite[Section 2.2]{ItoJinZou:2010}
for details.
\end{rmk}

The relation between the criteria $\Phi_\gamma$ and $\Psi_\gamma$ is made more precise in
the following theorem: $\Psi_\gamma$ always lies below $\Phi_\gamma$, and thus at each
local minimum, $\Phi_\gamma$ is sharper and numerically easier to locate.
\begin{thm}
For any $\gamma>0$, the following inequality holds
\begin{equation*}
\Psi_\gamma(\bs\eta)\leq \Phi_\gamma(\bs\eta),\quad \forall
\bs\eta>0.
\end{equation*}
The equality is achieved if and only if the balancing equation
\eqref{eqn:bal} is verified.
\end{thm}
\begin{proof}
Recall that for any $a,b,c\geq0$ and $p,q,r>1$ with $\frac{1}{p}+\frac{1}{q}+
\frac{1}{r}=1$, there holds the generalized Young's inequality $abc\leq
\frac{a^p}{p}+\frac{b^q}{q}+\frac{c^r}{r}$, with equality holds if and only if
$a^p=b^q=c^r$. Let $p=\frac{2+\gamma}{\gamma}$ and $q=r=2+\gamma$. Applying the
inequality with $a=\phi^{\frac{\gamma}{2+\gamma}}
(\eta_1\eta_2)^{-\frac{\gamma}{2(2+\gamma)}}$, $b=
(\gamma\psi_1)^\frac{1}{2+\gamma}(\eta_1\eta_2^{-1})^\frac{1}{2(2+\gamma)}$ and
$c=(\gamma\psi_2)^\frac{1}{2+\gamma} (\eta_1^{-1}\eta_2)^\frac{1}{2(2+\gamma)}$ gives
\begin{equation*}
\begin{aligned}
\phi^{\frac{\gamma}{2+\gamma}}\psi_1^{\frac{1}{2+\gamma}}\psi_2^{\frac{1}{2+\gamma}}
(\eta_1\eta_2)^{-\frac{\gamma}{2(2+\gamma)}}\gamma^\frac{2}{2+\gamma}
&\leq\frac{\gamma}{2+\gamma}\frac{\phi+\bs\eta\cdot\bs\psi}{(\eta_1\eta_2)^\frac{1}{2}}
=\frac{\gamma}{2+\gamma}\frac{F({\bs{\eta}})}{(\eta_1\eta_2)^\frac{1}{2}}.
\end{aligned}
\end{equation*}
Hence
\begin{equation*}
\phi^\frac{\gamma}{2+\gamma}\psi_1^\frac{1}{2+\gamma}\psi_2^\frac{1}{2+\gamma}\leq
\frac{\gamma^\frac{\gamma}{\gamma+2}}{2+\gamma}\frac{F(\bs\eta)}{(\eta_1\eta_2)^\frac{1}{2+\gamma}}.
\end{equation*}
Therefore, we have
\begin{equation*}
\Psi_\gamma(\bs\eta)\leq
\frac{\gamma^\gamma}{(2+\gamma)^{2+\gamma}}\frac{F^{2+\gamma}(\bs\eta)}{\eta_1\eta_2}=\Phi_\gamma(\bs\eta).
\end{equation*}
The equality holds if and only if $a^p=b^q=c^r$, i.e.,
\begin{equation*}
[\phi^{\frac{\gamma}{2+\gamma}}(\eta_1\eta_2)^{-\frac{\gamma}{
2(2+\gamma)}}]^\frac{2+\gamma}{\gamma}=[(\gamma\psi_1)^\frac{1}{2+\gamma}(\eta_1\eta_2^{-1})^\frac{1}{2(2+\gamma)}]^{2+\gamma}
=[(\gamma\psi_2)^\frac{1}{2+\gamma}(\eta_1^{-1}\eta_2)^\frac{1}{2(2+\gamma)}]^{2+\gamma}.
\end{equation*}
Simplifying this gives the balancing equation \eqref{eqn:bal}. This
concludes the proof.
\end{proof}

The following result shows an interesting property of a minimizer to Criterion
$\Phi_\gamma$.
\begin{thm}
At a local minimizer $\bs\eta^\ast$ to the function $\Phi_\gamma$, the partial
derivatives of $F(\bs\eta)$ exist.
\end{thm}
\begin{proof}
Assume that the assertion is not true, i.e., $\bs\eta^\ast$ is a discontinuity point of
at least one $\psi_i$. Since $\bs\eta^\ast$ is a local minimizer, we have
\begin{equation*}
\partial_i^-\Phi_\gamma(\bs\eta^\ast)\leq 0\quad \mbox{and}\quad
\partial_i^+\Phi_\gamma(\bs\eta^\ast)\geq0.
\end{equation*}
In particular, this implies that
$\partial_i^+\Phi_\gamma(\bs\eta^\ast)-\partial_i^-\Phi_\gamma(\bs\eta^\ast)\geq0$. Note
that
\begin{equation*}
\partial_i^+\Phi_\gamma(\bs\eta^\ast)-\partial_i^-\Phi_\gamma(\bs\eta^\ast)
=(2+\gamma)c_\gamma\frac{1}{\eta^\ast_1\eta^\ast_2}F^\gamma(\bs\eta^\ast)
\left[\partial_i^+F(\bs\eta^\ast)-\partial_i^-F(\bs\eta^\ast)\right]
\end{equation*}
and consequently $\partial_i^+F(\bs\eta^\ast)-\partial_i^-F(\bs\eta^\ast)\geq0$. This is
in contradiction with the fact that at a discontinuity point $\bs\eta^\ast$,
$\partial_i^+F(\bs\eta^\ast)- \partial_i^-F(\bs\eta^\ast)<0$ by the monotonicity of the
function $\psi_i(x_{\bs\eta}^\delta)$ with respect to $\eta_i$.
\end{proof}

Now we present an a posteriori error estimate for Criterion $\Phi_\gamma$ when $Y$ is a
Hilbert space and $\phi(x,y^\delta)=\|Kx-y^\delta\|^2$ and convex penalties. The proof
will be presented elsewhere, and we also refer to \cite{Ito:2010a}. Theorem
\ref{thm:errest} provides one a posteriori way to check the automatically determined
(vector-valued) regularization parameter, and partially justifies the criterion
theoretically.

\begin{thm}\label{thm:errest}
Let the following source condition be satisfied for the exact solution $x^\dagger$: for
any $t\in[0,1]$ there exists a $w_t\in Y$
\begin{equation*}
 \xi_t\in\partial\left(t\psi_1(x^\dagger)+(1-t)\psi_2(x^\dagger)\right)\quad \mbox{and}\quad
 \xi_t=K^\ast w_t.
\end{equation*}
Then for every $\bs\eta^\ast$ determined by the criterion $\Phi_\gamma$, there exists
some constant $C$ such that
\begin{equation*}
  d_{\xi_{t^\ast}}(x_{\bs\eta^\ast}^\delta,x^\dagger)\leq C\left(\|w_{t^\ast}\|+\frac{F^{1+\frac{\gamma}{2}}
  (\delta\bs e)}{F^{1+\frac{\gamma}{2}}(\bs\eta^\ast)}\right)\max(\delta,\delta_\ast),
\end{equation*}
where $\bs e=(1,1)^\mathrm{T}$, $\delta_\ast=\|Kx_{\bs\eta^\ast}^\delta-y^\delta\|$, and
$t^\ast=\eta_1^\ast/(\eta_1^\ast+\eta_2^\ast)$.
\end{thm}

Finally, we present two algorithms, see Algorithms \ref{alg:fpai} and \ref{alg:fpaii},
for computing a minimizer of Criterion $\Phi_\gamma$. The algorithms are of fixed point
type, and can be regarded as natural extensions of the fixed point algorithm in
\cite{Ito:2010a}. Practically, the algorithms merit a very steady and fast convergence.

\begin{algorithm}[t]
	\caption{Fixed point algorithm I.\label{alg:fpai}}
\begin{algorithmic}[1]
	\State Choose $\gamma$, $\bs\eta^0$ and set $k=0$.
	\Repeat
	\State Solve for $x^{k+1}$ by the Tikhonov regularization method
        \begin{equation*}
            x^{k+1}=\arg\min_{x}\left\{\phi(x,y^\delta)+\bs\eta^k\cdot\bs\psi(x)\right\}.
        \end{equation*}
	\State Update the regularization parameter $\bs\eta^{k+1}$ by
        \begin{equation*}
           \begin{aligned}
               \eta_1^{k+1}&=\frac{1}{1+\gamma}\frac{\phi(x^{k+1},y^\delta)+\eta_2^k\psi_2(x^{k+1})}{\psi_1(x^{k+1})},\\
               \eta_2^{k+1}&=\frac{1}{1+\gamma}\frac{\phi(x^{k+1},y^\delta)+\eta_1^k\psi_1(x^{k+1})}{\psi_2(x^{k+1})}.
           \end{aligned}
        \end{equation*}
	\Until A stopping criterion is satisfied.
\end{algorithmic}
\end{algorithm}

\begin{algorithm}[t]
	\caption{Fixed point algorithm II.\label{alg:fpaii}}
\begin{algorithmic}[1]
	\State Choose $\gamma$, $\bs\eta^0$ and set $k=0$.
	\Repeat
	\State Solve for $x^{k+1}$ by the Tikhonov regularization method
        \begin{equation*}
            x^{k+1}=\arg\min_{x}\left\{\phi(x,y^\delta)+\bs\eta^k\cdot\bs\psi(x)\right\}.
        \end{equation*}
	\State Update the regularization parameter $\bs\eta^{k+1}$ by
        \begin{equation*}
           \begin{aligned}
               \eta_1^{k+1}&=\frac{1}{\gamma}\frac{\phi(x^{k+1},y^\delta)}{\psi_1(x^{k+1})},\\
               \eta_2^{k+1}&=\frac{1}{\gamma}\frac{\phi(x^{k+1},y^\delta)}{\psi_2(x^{k+1})}.
           \end{aligned}
        \end{equation*}
	\Until A stopping criterion is satisfied.
\end{algorithmic}
\end{algorithm}

\section{Numerical experiments}\label{sec:exp}

This part presents numerical results for three examples, which are integral equations of
the first kind with kernel $k(s,t)$ and solution $x^\dagger(t)$, to illustrate features
of multi-parameter models. The discretized linear system takes the form
$\mathbf{Kx}^\dagger=\mathbf{y}^\dagger$. The data $\mathbf{y}^\dagger$ is corrupted by
noises, i.e., $y_i^\delta=y_i^\dagger+\max_{i}\{|y_i^\dagger|\}\varepsilon \xi_i$, where
$\xi_i$ are standard Gaussian variables and $\varepsilon$ refers to the relative noise
level. The fidelity $\phi$ is taken to be the standard least-squares fitting. We present
only the numerical results for Algorithm II, as Algorithm I exhibits similar convergence
behavior. The initial guess is always taken to be $1\times10^{-3}$, and it is stopped if
the relative change of $\bs\eta$ is smaller than $1.0\times10^{-3}$. The parameter
$\gamma$ in Criterion $\Phi_\gamma$ is determined by a two-step procedure
\cite{Ito:2010a}: The initial guess for $\gamma$ is set to $5$, and then it is
automatically adjusted according to the estimate noise level.

\subsection{$H^1$-$TV$ model}
\begin{exam}\label{exam:h1tv}
Let $\zeta(t)=\chi_{|t|\leq3}(1+\cos\frac{\pi t}{3})$, and the kernel $k$ is given by
$k(s,t)=\zeta(s-t)$. The exact solution $x^\dagger$ is shown in Fig. \ref{fig:h1tv}, and
the integration interval is $[-6,6]$. The solution $x^\dagger$ exhibits both flat and
smoothly varying regions, and thus we adopt two penalties
$\psi_1(x)=\frac{1}{2}|x|_{H^1}^2$ and $\psi_2(x)=|x|_{TV}$ for preserving their distinct
features. The size of the problem is 100.
\end{exam}

\begin{table}
\centering \caption{Numerical results for Example \ref{exam:h1tv}.}
\begin{tabular}{ccccccccc}
\hline
$\epsilon$&$\bs\eta_\mathrm{b}$&$\bs\eta_\mathrm{o}$&$\eta_\mathrm{h1}$&$\eta_\mathrm{tv}$&$e_\mathrm{b}$&$e_\mathrm{o}$&$e_\mathrm{h1}$&$e_\mathrm{tv}$\\
\hline
5e-2&(3.44e-3,5.75e-3)&(2.36e-4,2.14e-3)&5.68e-4&9.27e-3&3.31e-2&2.66e-2&3.97e-2&1.07e-1\\
5e-3&(1.03e-4,1.83e-4)&(2.19e-5,3.70e-4)&6.81e-5&4.85e-4&2.27e-2&1.10e-2&2.69e-2&9.48e-2\\
5e-4&(3.32e-6,6.12e-6)&(2.89e-6,5.07e-5)&1.26e-6&6.08e-5&1.25e-2&8.85e-3&1.38e-2&4.48e-2\\
5e-5&(1.07e-7,2.04e-7)&(7.04e-8,5.23e-6)&1.14e-7&4.06e-6&6.82e-3&5.53e-3&9.40e-3&1.68e-2\\
5e-6&(3.01e-9,5.77e-9)&(2.06e-10,6.65e-9)&6.01e-10&2.24e-7&4.50e-3&2.89e-3&5.28e-3&5.12e-3\\
\hline
\end{tabular}\label{tab:h1tv}
\end{table}

\begin{figure}
\centering
\begin{tabular}{ccc}
\includegraphics[width=5.2cm]{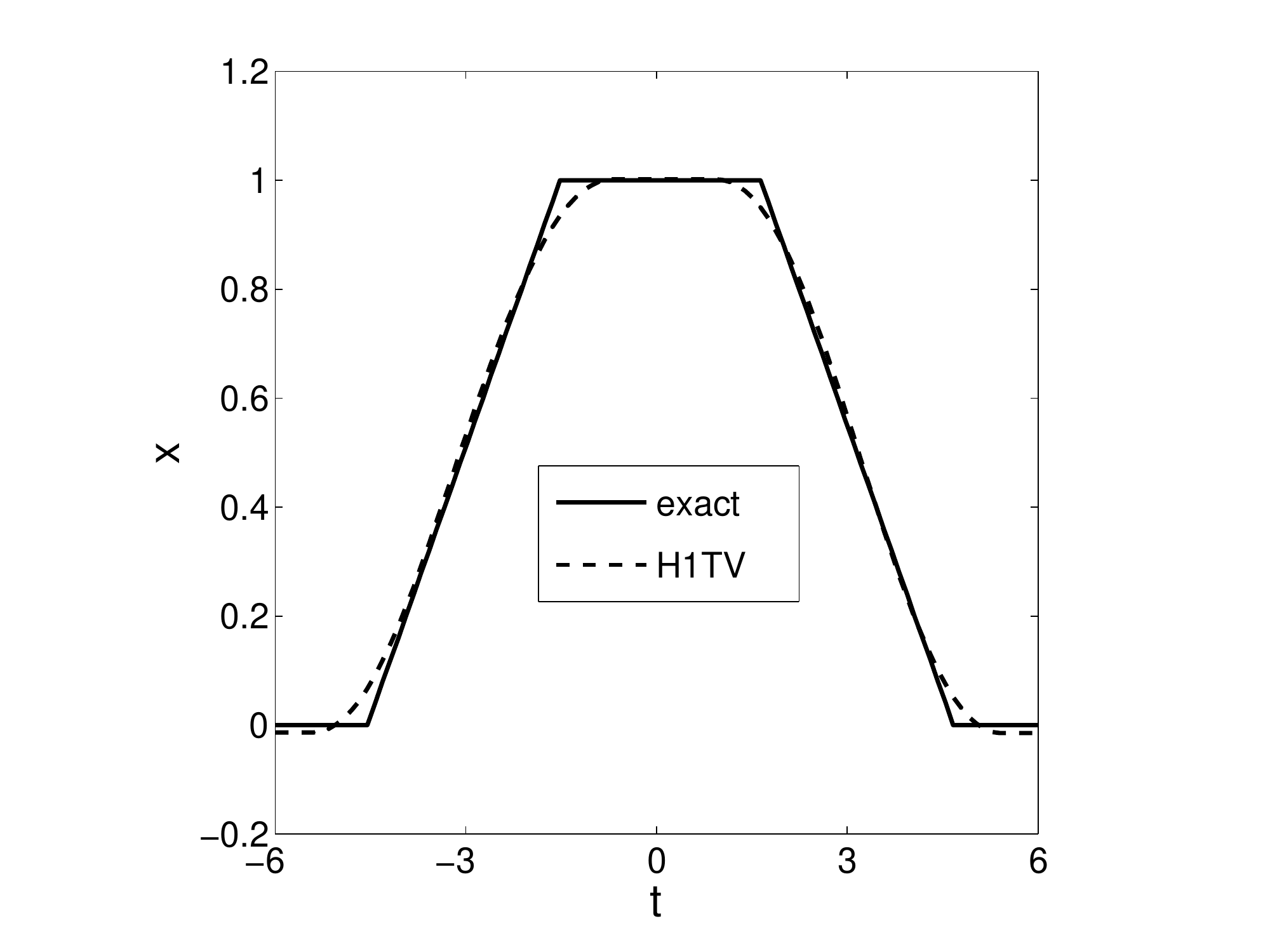}&\includegraphics[width=5.2cm]{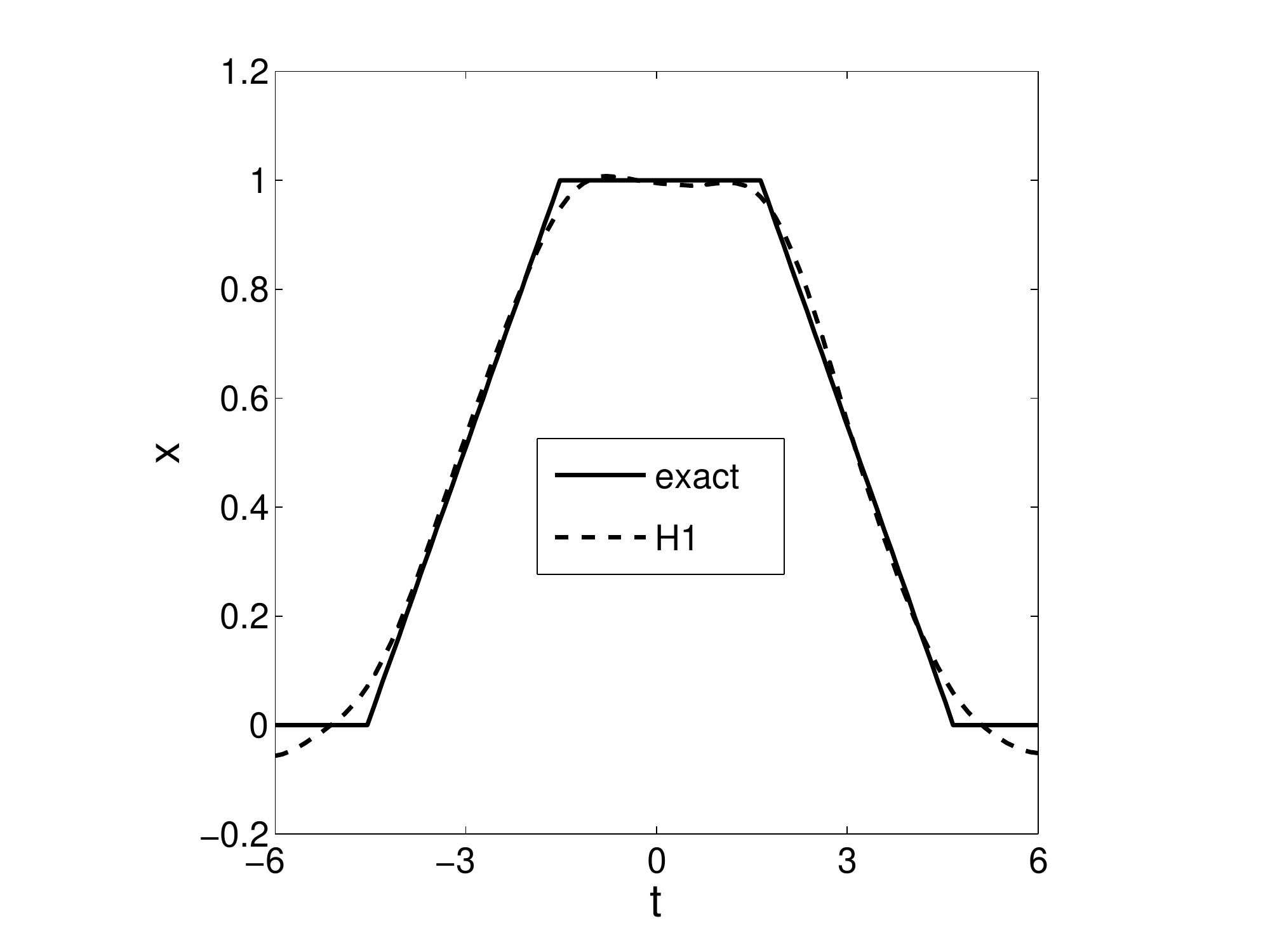}&\includegraphics[width=5.2cm]{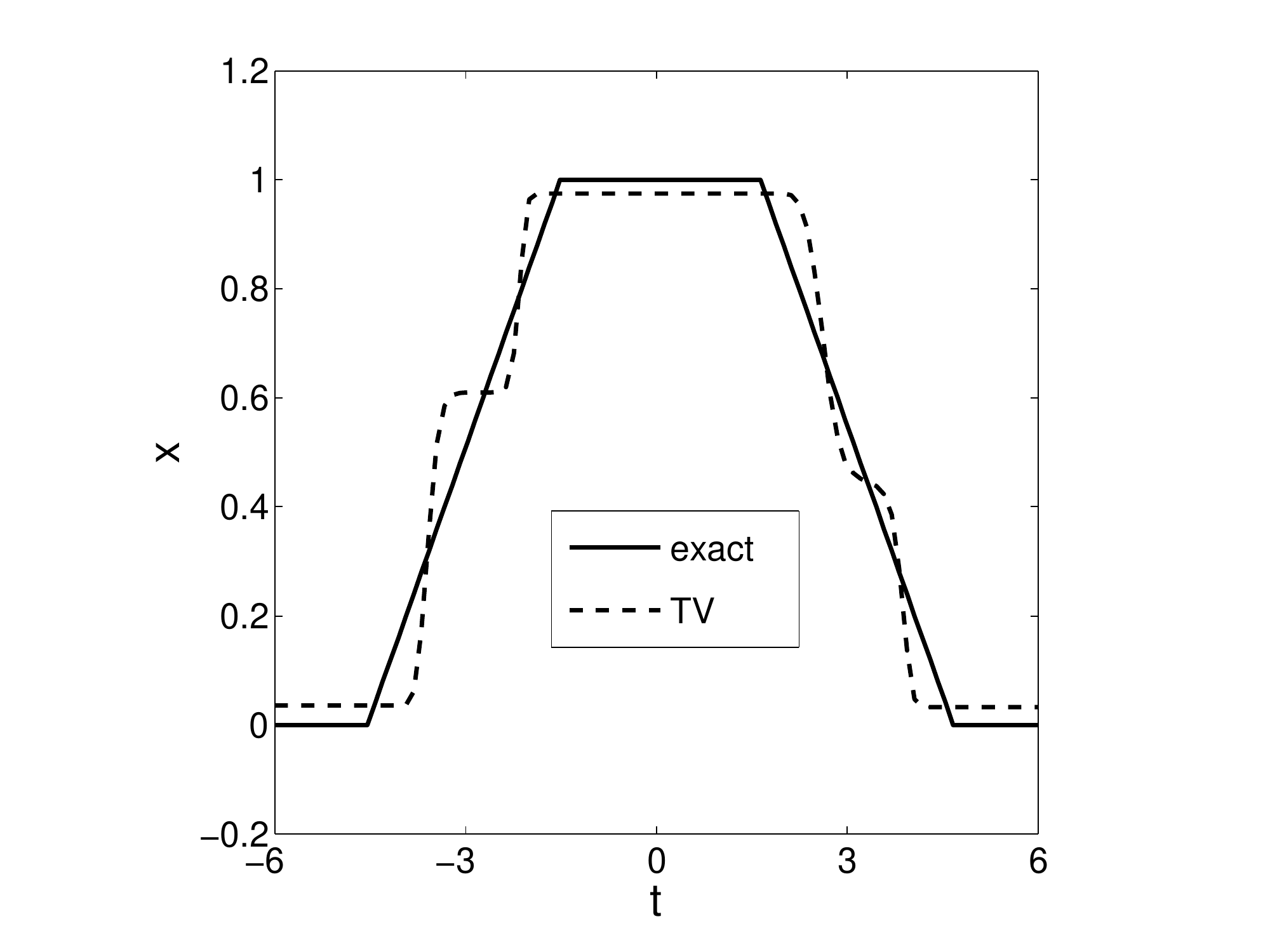}\\
$H^1$-$TV$ sol. with $\bs\eta_\mathrm{b}$&$H^1$ sol. with $\eta_\mathrm{h1}$ & $TV$ sol. with $\eta_\mathrm{tv}$
\end{tabular}
\caption{Numerical results for Example \ref{exam:h1tv} with $5\%$ noise.}\label{fig:h1tv}
\end{figure}

The numerical results are summarized in Table \ref{tab:h1tv}. In the table, the
subscripts $\mathrm{b}$ and $\mathrm{o}$ refer to the balancing principle and the optimal
choice, i.e., the value giving the smallest reconstruction error, respectively. The
results for single-parameter models are indicated by subscripts $\mathrm{h1}$ and
$\mathrm{tv}$, and the respective penalty parameter shown in Table \ref{tab:h1tv} is the
optimal one. The accuracy of the results is measured by the relative $L^2$ error
$e=\|x-x^\dagger\|/\|x^\dagger\|$. A first observation is that the error $e_\mathrm{b}$,
by the balancing principle for the proposed model $H^1$-$TV$ is smaller than the optimal
choice for either $H^1$ or $TV$ penalty. This illustrates clearly the benefit of using
multi-parameter model. Interestingly, the balancing principle gives an error fairly close
to the optimal one, and the error decreases as the noise level decreases.

\begin{figure}
\centering
\begin{tabular}{cc}
\includegraphics[width=6cm]{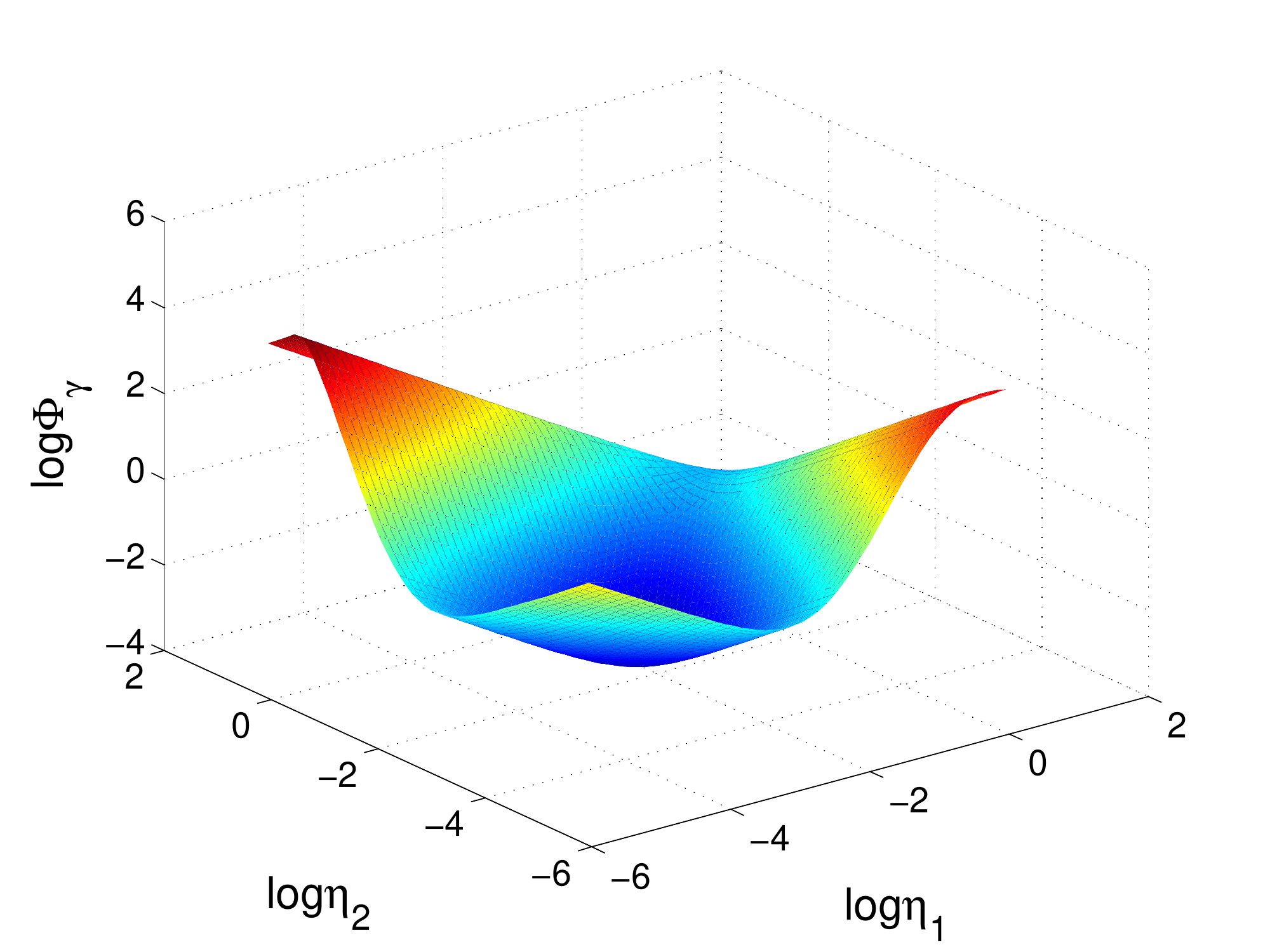}&\includegraphics[width=6cm]{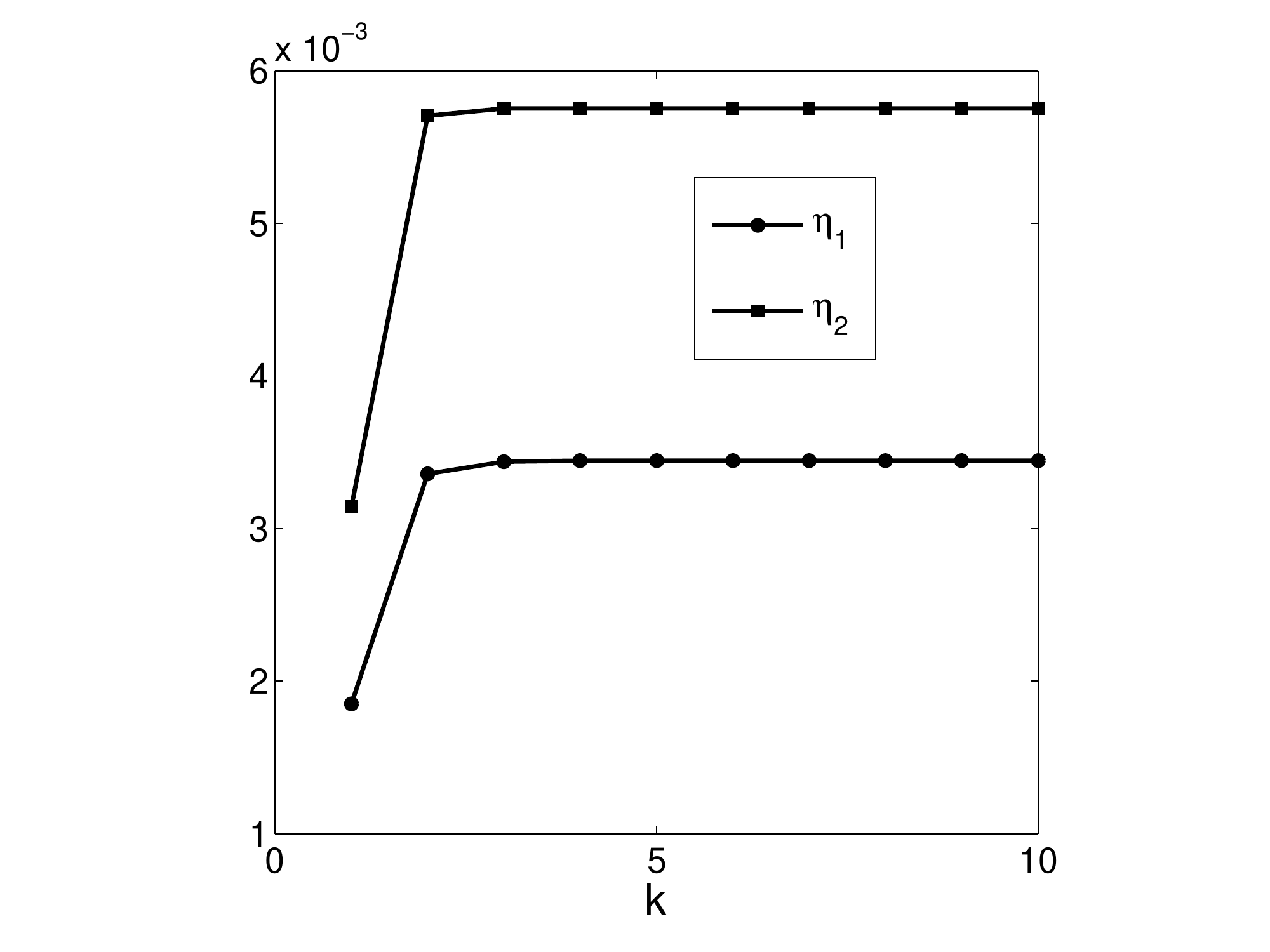}\\
$\Phi_\gamma(\bs\eta)$& convergence of Algorithm II
\end{tabular}
\caption{Numerical results for Example \ref{exam:h1tv} with $5\%$ noise.}\label{fig:h1tv:alg}
\end{figure}

The numerical results for Example 1 with $\varepsilon=5\%$ noise is shown in Fig.
\ref{fig:h1tv}. In particular, the classical $H^1$ smoothness penalty fails to restore
the flat region satisfactorily, whereas the $TV$ approach suffers from stair-case effect
in the gray region and reduced magnitude in the flat region, see Fig. \ref{fig:h1tv}. In
contrast, the proposed $H^1$-$TV$ model can preserve the magnitude of flat region while
reconstruct the gray region excellently. Therefore, it indeed combines the strengths of
both $H^1$ and $TV$ models, and is suitable for restoring images with both flat and gray
regions. The criterion $\Phi_\gamma$ is numerically well-behaved: there is a distinct
local minimum, and it is numerically easy to minimize, see Fig. \ref{fig:h1tv:alg}.
Finally, we would like to remark that the algorithm converge rapidly with the convergence
achieved in five iterations, see Fig. \ref{fig:h1tv:alg}.

\subsection{$\ell^1$-$\ell^2$ model}
\begin{exam}\label{exam:l1l2}
The kernel $k$ is given by $k(s,t)=\tfrac{1}{4}\left(\tfrac{1}{16}+ (s-t)^2\right)^{
-\frac{3}{2}}$, the exact solution $x^\dagger$ consists of two bumps and it is shown in
Fig. \ref{fig:l1l2}. The penalties are $\psi_1(x)=\|x\|_{\ell^1}$ and
$\psi_2(x)=\frac{1}{2} \|x\|_{\ell^2}^2$ to retrieve the groupwise sparsity structure.
The integration interval is $[0,1]$. The size of the problem is 100.
\end{exam}

\begin{figure}
\centering
\begin{tabular}{ccc}
\includegraphics[width=5.2cm]{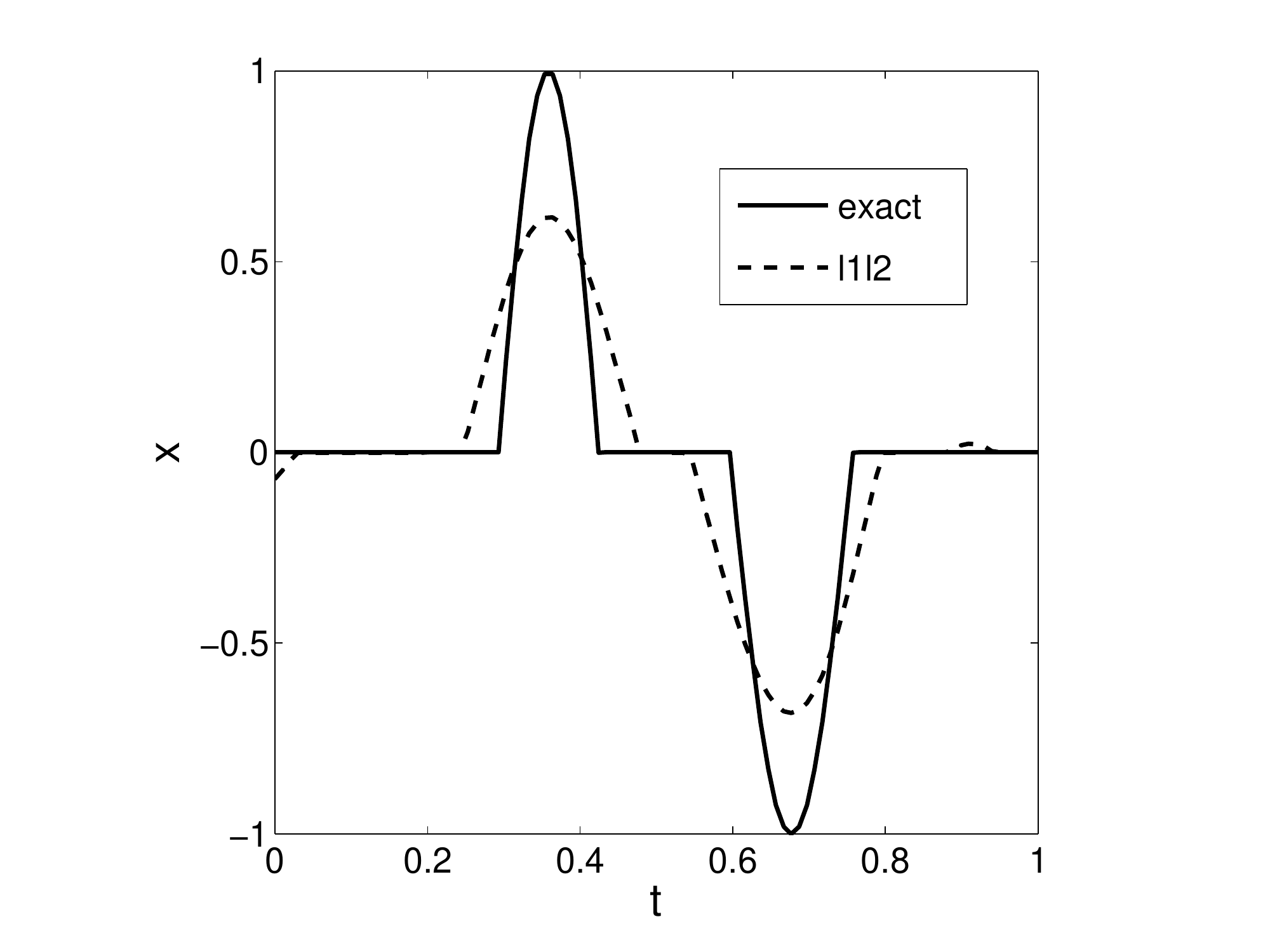}&
\includegraphics[width=5.2cm]{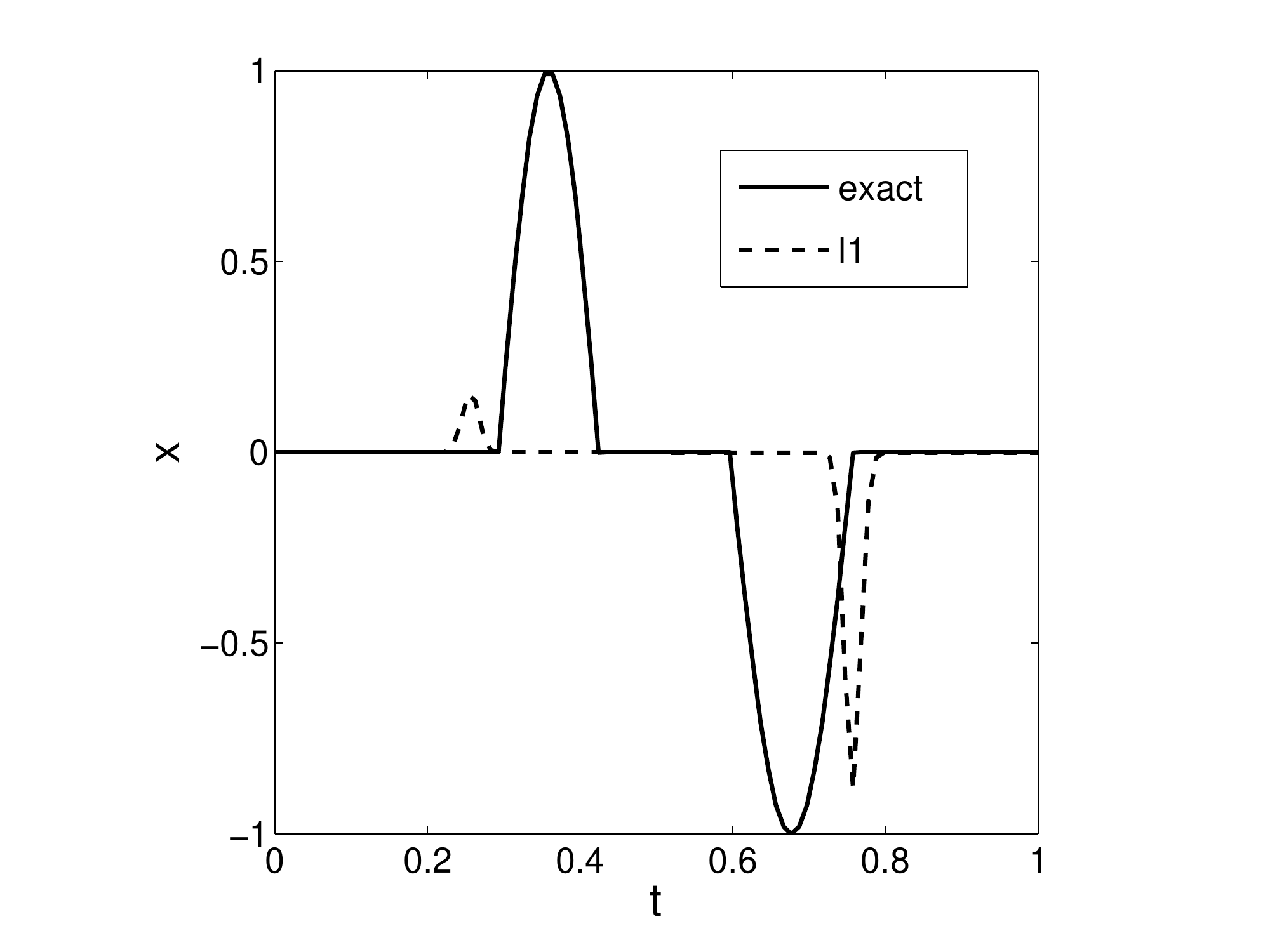}&
\includegraphics[width=5.2cm]{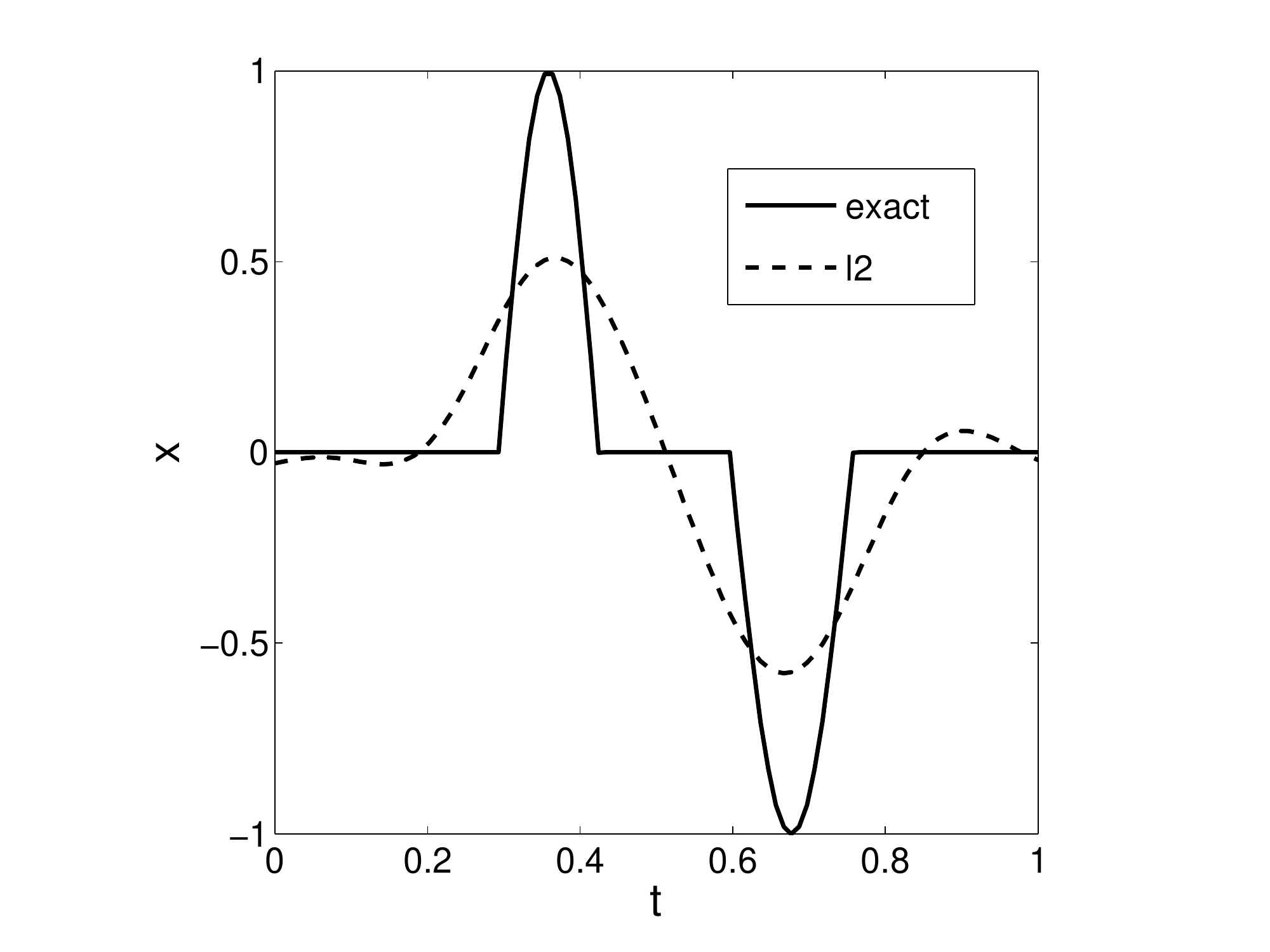}\\
$\ell^1$-$\ell^2$ sol. with $\bs\eta_\mathrm{b}$&$\ell^1$ sol. with $\eta_\mathrm{l1}$& $\ell^2$ sol. with $\eta_\mathrm{l2}$
\end{tabular}
\caption{Numerical results for Example \ref{exam:l1l2} with $5\%$ noise.}\label{fig:l1l2}
\end{figure}

\begin{figure}
\centering
\begin{tabular}{cc}
\includegraphics[width=6cm]{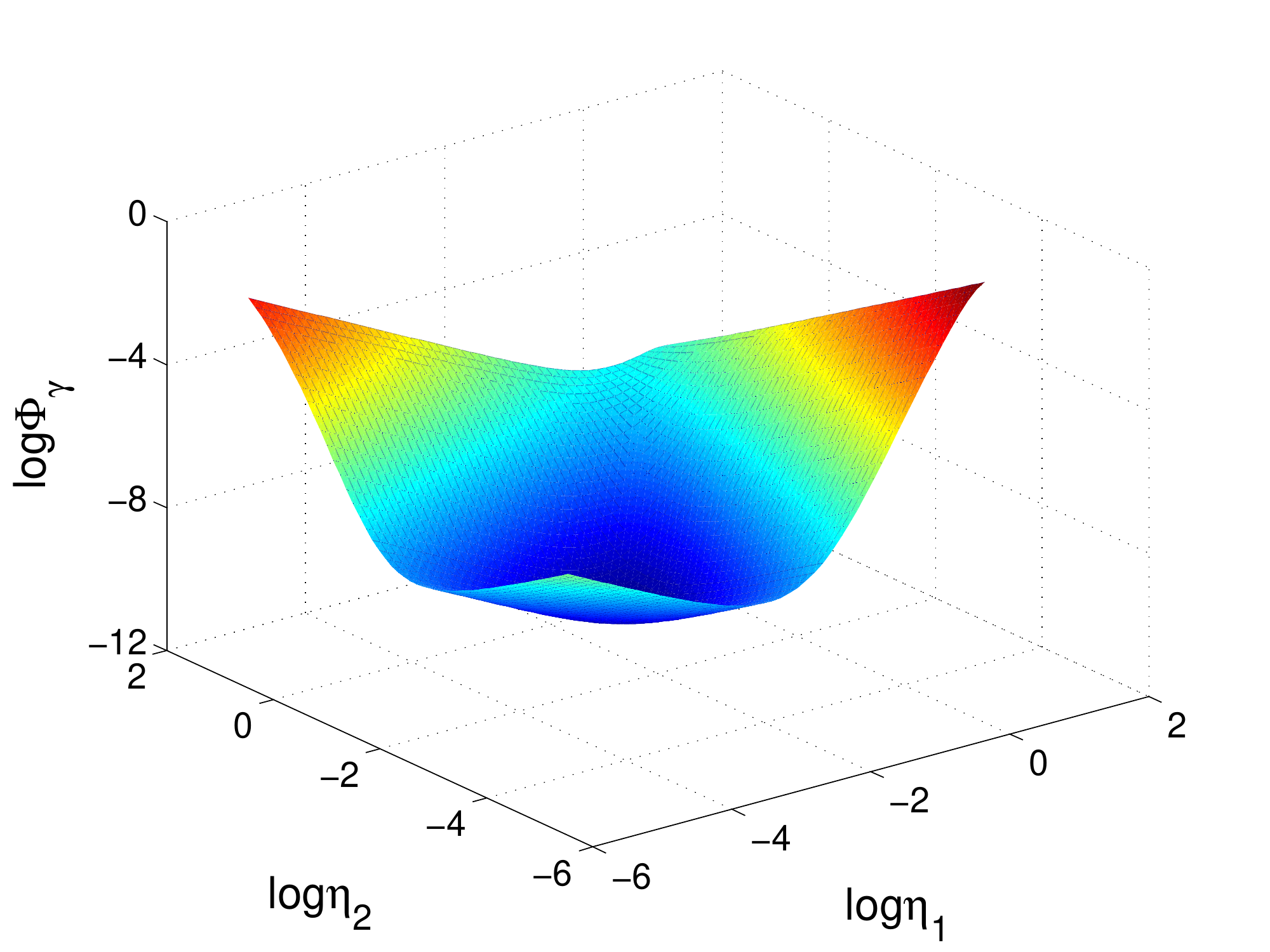}&\includegraphics[width=6cm]{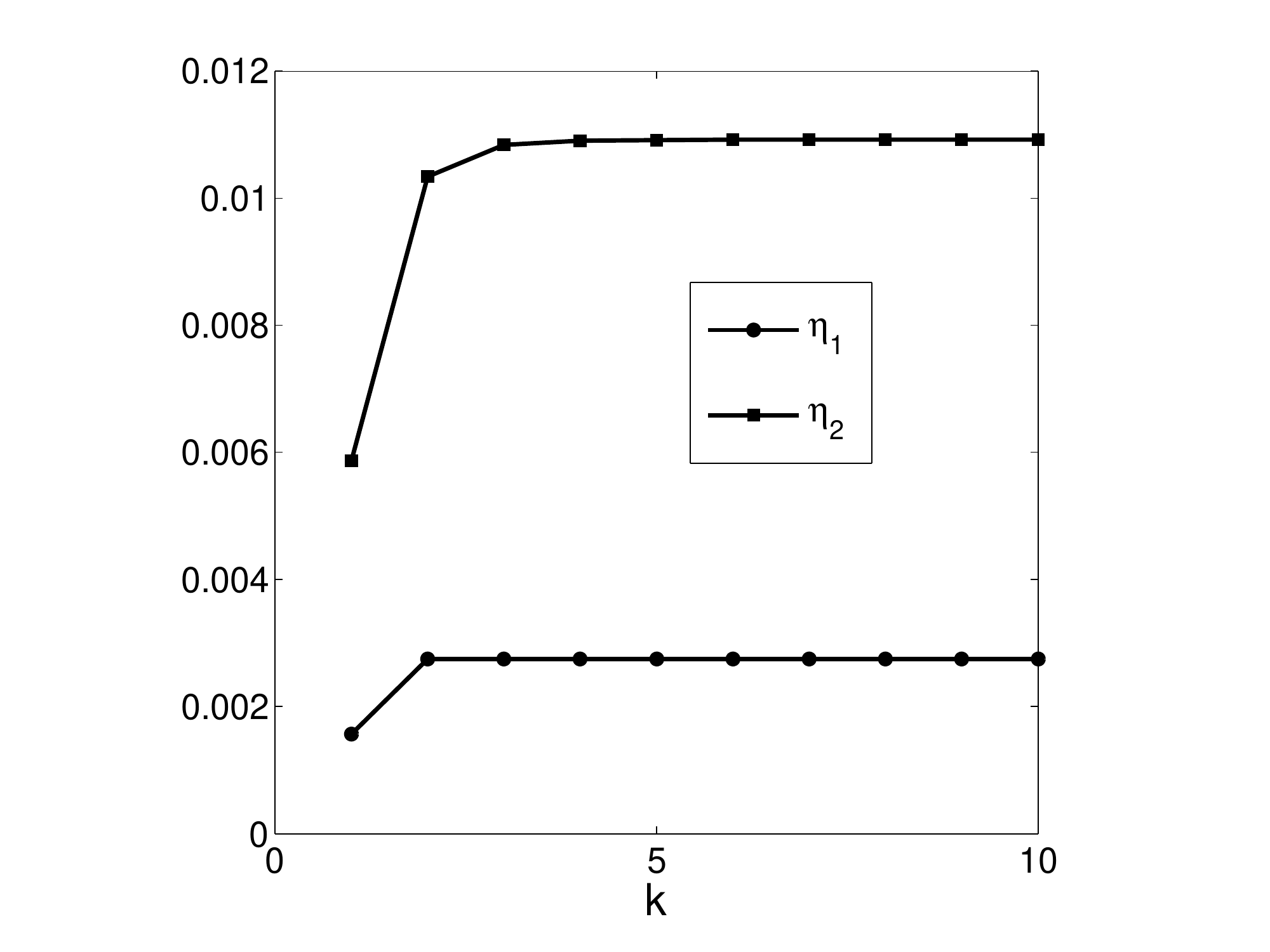}   \\
$\Phi_\gamma(\bs\eta)$& convergence of Algorithm II
\end{tabular}
\caption{Numerical results for Example \ref{exam:l1l2} with $5\%$ noise.}\label{fig:l1l2:alg}
\end{figure}

\begin{table}
\centering\caption{Numerical results for Example \ref{exam:l1l2}.}
\begin{tabular}{ccccccccc}
\hline
$\epsilon$&$\bs\eta_\mathrm{b}$&$\bs\eta_\mathrm{o}$&$\eta_\mathrm{l1}$&$\eta_\mathrm{l2}$&$e_\mathrm{b}$&$e_\mathrm{o}$&$e_\mathrm{l1}$&$e_\mathrm{l2}$\\
\hline
5e-2&(2.75e-3,1.09e-2)&(3.16e-3,1.32e-3)&2.96e0 &3.34e-3&4.18e-1&8.72e-2&1.04e0 &4.59e-1\\
5e-3&(9.16e-5,2.86e-4)&(2.46e-4,1.07e-4)&1.03e-4&3.06e-5&2.09e-1&1.24e-2&8.97e-1&2.90e-1\\
5e-4&(2.82e-6,7.48e-6)&(2.34e-5,1.14e-5)&1.30e-5&4.08e-6&5.76e-2&7.98e-3&6.18e-1&2.17e-1\\
5e-5&(8.89e-8,2.26e-7)&(2.27e-6,1.06e-6)&1.24e-6&3.84e-8&1.57e-2&4.71e-3&4.85e-1&1.66e-1\\
5e-6&(2.79e-9,7.07e-9)&(1.66e-7,1.03e-7)&4.12e-9&1.41e-9&1.27e-2&2.27e-3&2.61e-1&9.55e-2\\
\hline
\end{tabular}\label{tab:l1l2}
\end{table}

The numerical results for this example are show in Table \ref{tab:l1l2} and Fig.
\ref{fig:l1l2}. Here we are interested in the group structure of the solution with
minimal number of influencing factors (nonzero coefficients). Again, we observe that the
elastic-net compares favorably with the conventional $\ell^1$ and $\ell^2$ penalties in
terms of the error, and the balancing principle can give reasonable estimate for the
optimal choice. The conventional $\ell^2$ solution contains almost no zero entries, and
thus it fails to distinguish between influencing and noninfluencing coefficients, i.e.,
identifying relevant factors. This difficulty is partially remedied by the $\ell^1$ model
in that many entries of the $\ell^1$ solution are zero. Therefore, some relevant factors
are correctly identified. However, it tends to select only some instead of all relevant
factors, i.e., group structure. The elastic-net model combines the best of both $\ell^1$
and $\ell^2$ models, and it achieves the desired goal of identifying the group structure.
Moreover, the magnitude assigned to the coefficients are reasonable compared to others.
The algorithm converges quickly within five iterations.

\subsection{2D image deblurring}
\begin{exam}\label{exam:l1l22d} The penalties are
$\psi_1(x)=\|x\|_{\ell^1}$ and $\psi_2(x)=\frac{1}{2} \|x\|_{\ell^2}^2$. The kernel $k$
performs standard Gaussian blur with standard deviation $1$ and blurring width $5$. The
exact solution $x^\dagger$ is shown in Fig. \ref{fig:l1l22d}. The size of the image is
$50\times50$.
\end{exam}

\begin{figure}\centering

\begin{tabular}{ccc}
\includegraphics[width=3.5cm]{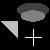}&\includegraphics[width=3.5cm]{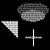}&\\
exact $x^\dagger$& $\ell^1$-$\ell^2$ sol. with $\bs\eta_\mathrm{o}=$(1.25e-2,1.29e-3) & \\
\includegraphics[width=3.5cm]{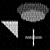}&\includegraphics[width=3.5cm]{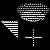}
&\includegraphics[width=3.5cm]{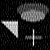}\\
$\ell^1$-$\ell^2$ sol. with $\bs\eta_\mathrm{b}=$(1.14e-3,
1.12e-3)& $\ell^1$ sol. with $\eta_\mathrm{l1}$=5.30e-1 & $\ell^2$ sol. with $\eta_\mathrm{l2}=$3.31e-3
\end{tabular}
\caption{Numerical results for Example \ref{exam:l1l22d} with $1\%$
noise.}\label{fig:l1l22d}
\end{figure}

This example showcases a more realistic problem of image deblurring. Here one half of the
data points are retained. The reconstructions for $1\%$ noise are shown in Fig.
\ref{fig:l1l22d}. The $\ell^1$ solution is more spiky, and neighboring pixels more or
less act independently. In particular, due to missing data, there are some missing pixels
in the blocks and the cross to be recovered. In contrast, the $\ell^2$ solution is more
blockwise, but there are many nonzero coefficients indicated by the small spurious
oscillations in the background. The elastic-net model achieves the best of the two:
retaining the block structure with only fewer spurious nonzero coefficients. This is
deemed important in medical imaging, e.g., classification. The numbers are more telling:
$e_\mathrm{b}=2.99\times10^{-1}$, $e_\mathrm{o}=2.44\times10^{-1}$,
$e_\mathrm{l1}=9.21\times10^{-1}$, and $e_\mathrm{l2}=3.42\times10^{-1}$. Therefore, the
error $e_\mathrm{b}$ for elastic-net agrees well with the optimal choice, and it is
smaller than that with the optimal choices for both $\ell^1$ and $\ell^2$ models.

\section{Concluding remarks}
We have studied theoretical properties of multi-parameter Tikhonov regularization. Some
properties, e.g., monotonicity, concavity, asymptotic and differentiability, of the value
function, were established. The discrepancy principle is partially justified in terms of
consistency and convergence rates, however, the regularization parameter is not uniquely
determined, which partially limits its practical application. It is of interest to
develop auxiliary rules, which is currently under investigation. In contrast, the
balancing principle allows justifications in terms of a posteriori error estimate and
efficient numerical implementation. The numerical experiments show that multi-parameter
models can significantly improve the reconstruction quality and the balancing principle
can give reasonable results in comparison with the optimal choice in a computationally
efficient way. The two proposed algorithms for computing the parameters of the balancing
principle deliver excellent convergence behavior. However, a rigorous convergence
analysis remains to be established.

\section*{Acknowledgements}
This work was partially carried out during the visit of Bangti Jin at Department of
Mathematics, North Carolina State University. He would like to thank Professor Kazufumi
Ito for hospitality. Bangti Jin was supported by Alexander von Humboldt foundation
through a postdoctoral research fellowship.

\bibliographystyle{abbrv}
\bibliography{multitikh}

\begin{thebibliography}{10}

\bibitem{BelgeKilmerMiller:2002}
M.~Belge, M.~E. Kilmer, and E.~L. Miller.
\newblock Efficient determination of multiple regularization parameters in a
  generalized {L}-curve framework.
\newblock {\em Inverse Problems}, 18(4):1161--1183, 2002.

\bibitem{BrezinskiRedivoRodriguezSeatzu:2003}
C.~Brezinski, M.~Redivo-Zaglia, G.~Rodriguez, and S.~Seatzu.
\newblock Multi-parameter regularization techniques for ill-conditioned linear
  systems.
\newblock {\em Numer. Math.}, 94(2):203--228, 2003.

\bibitem{BrooksAhmadMacLeodMaratos:1999}
D.~H. Brooks, G.~F. Ahmad, R.~S. Macleod, and G.~M. Maratos.
\newblock Inverse electrocardiography by simultaneous imposition of multiple
  constraints.
\newblock {\em IEEE Trans. Biomed. Eng.}, 46(1):3--18, 1999.

\bibitem{ChenLuXuYang:2008}
Z.~Chen, Y.~Lu, Y.~Xu, and H.~Yang.
\newblock Multi-parameter {T}ikhonov regularization for linear ill-posed
  operator equations.
\newblock {\em J. Comput. Math.}, 26(1):37--55, 2008.

\bibitem{DeMolVetoRosasco:2009}
C.~De~Mol, E.~De~Vito, and L.~Rosasco.
\newblock Elastic-net regularization in learning theory.
\newblock {\em J. Complexity}, 25(2):201--230, 2009.

\bibitem{EnglHankeNeubauer:1996}
H.~W. Engl, M.~Hanke, and A.~Neubauer.
\newblock {\em Regularization of {I}nverse {P}roblems}.
\newblock Kluwer, Dordrecht, 1996.

\bibitem{Ito:2010a}
K.~Ito, B.~Jin, and T.~Takeuchi.
\newblock A regularization parameter for nonsmooth {T}ikhonov regularization.
\newblock Technical Report UTMS 2010-3, Graduate School of Mathematical
  Sciences, University of Tokyo, 2010.

\bibitem{ItoJinZou:2010}
K.~Ito, B.~Jin, and J.~Zou.
\newblock A new choice rule for regularization parameters in {T}ikhonov
  regularization.
\newblock {\em Appl. Anal.}, page in press, 2010.

\bibitem{JinLorenzSchiffler:2009}
B.~Jin, D.~A. Lorenz, and S.~Schiffler.
\newblock Elastic-net regularization: error estimates and active set methods.
\newblock {\em Inverse Problems}, 25(11):115022 (26pp), 2009.

\bibitem{JinZou:2009}
B.~Jin and J.~Zou.
\newblock Augmented {T}ikhonov regularization.
\newblock {\em Inverse Problems}, 25(2):025001, 25, 2009.

\bibitem{JinZou:2010}
B.~Jin and J.~Zou.
\newblock Iterative parameter choice by discrepancy principle.
\newblock Technical Report 2010-02(369), Department of Mathematics, Chinese
  University of Hong Kong, 2010.

\bibitem{LuPereverzev:2009}
S.~Lu and S.~V. Pereverzev.
\newblock Multi-parameter regularization and its numerical regularization.
\newblock Numer. Math., doi: 10.1007/s00211-010-0318-3.

\bibitem{LuPereverzevShaoTautenhahn:2010}
S.~Lu, S.~V. Pereverzev, Y.~Shao, and U.~Tautenhahn.
\newblock Discrepancy curves for multi-parameter regularization.
\newblock {\em J. Inv. Ill-Posed Probl.}, 18(6):655--676, 2010.

\bibitem{LuShenXu:2007}
Y.~Lu, L.~Shen, and Y.~Xu.
\newblock Multi-parameter regularization methods for high-resolution image
  reconstruction with displacement errors.
\newblock {\em IEEE Trans. Circuits Syst. I. Regul. Pap.}, 54(8):1788--1799,
  2007.

\bibitem{Mathe:2006}
P.~Math\'{e}.
\newblock The {L}epskii principle revisited.
\newblock {\em Inverse Problems}, 22(3):L11--L15, 2006.

\bibitem{Morozov:1966}
V.~A. Morozov.
\newblock On the solution of functional equations by the method of
  regularization.
\newblock {\em Sov. Math. Dokl.}, 7:414--417, 1966.

\bibitem{Stephanakis:1997}
I.~M. Stephanakis.
\newblock Regularized image restoration in multiresolution spaces.
\newblock {\em Opt. Eng.}, 36(6):1738--1744, 1997.

\bibitem{TikhonovArsenin:1977}
A.~N. Tikhonov and V.~Y. Arsenin.
\newblock {\em Solutions of {I}ll-{P}osed {P}roblems}.
\newblock John Wiley, New York, 1977.

\bibitem{XuFukudaLiu:2006}
P.~Xu, Y.~Fukuda, and Y.~Liu.
\newblock Multiple parameter regularization: numerical solutions and
  applications to the determination of geopotential from precise satellite
  orbits.
\newblock {\em J. Geod.}, 80(1):17--27, 2006.

\bibitem{ZouHastie:2005}
H.~Zou and T.~Hastie.
\newblock Regularization and variable selection via the elastic net.
\newblock {\em J. R. Stat. Soc. Ser. B}, 67(2):301--320, 2005.

\end{thebibliography}
\end{document}